\documentclass[a4paper,11pt,leqno]{article}

\usepackage{amsmath,amssymb,eucal,amsthm}
\usepackage{latexsym,bm}
\usepackage{mathrsfs,amsmath,amssymb,secdot}
\usepackage{amscd}

\DeclareMathOperator{\Hom}{\mathrm{Hom}}

\DeclareMathOperator{\Func}{\mathfrak{F}}

\def\bs#1{\mathbf{#1}}
\def\fs{\Psi}
\def\vv{\tau}
\def\la{\varepsilon}

\def\prs{\Delta_+}

\def\GL{\mathrm{GL}}

\def\abs#1{\lvert#1\rvert}
\long\def\comment#1{}
\def\norm#1{\lVert#1\rVert}
\allowdisplaybreaks

\newtheorem{theorem}{Theorem}[section]
\newtheorem{lemma}[theorem]{Lemma}
\newtheorem{proposition}[theorem]{Proposition}

\theoremstyle{definition}

\newtheorem{example}[theorem]{Example}

\newtheorem{remark}[theorem]{Remark}

\DeclareFontEncoding{OT2}{}{}
\DeclareFontSubstitution{OT2}{cmr}{m}{n}
\DeclareSymbolFont{cyss}{OT2}{wncyss}{m}{n}
\DeclareMathSymbol{\sh}{\mathbin}{cyss}{`x}

\numberwithin{equation}{section}

\begin{document}

\title{Zeta-functions of weight lattices of compact connected semisimple Lie groups}

\author{By \textit{Yasushi Komori} at Tokyo, \textit{Kohji Matsumoto} at Nagoya \\
and \textit{Hirofumi Tsumura} at Tokyo}

\date{}

\maketitle

\begin{abstract}
  We define zeta-functions of weight lattices 
of 
compact connected semisimple Lie groups. 
  If the group is simply-connected, these zeta-functions coincide with 
  ordinary zeta-functions of root systems of associated Lie algebras.
  In this paper we consider the general connected (but not necessarily
  simply-connected) case,    
  prove the explicit form of Witten's volume formulas for these zeta-functions,
  and further prove functional 
  relations among them which include their volume formulas.
  Also we give new examples of zeta-functions for which parity results hold.
\end{abstract}

%%%%%%%%%%%%%%%%%%%%%%%%%%%%%%%%%%%%%%%%%%%%                                            
\section{The background and the motivation} \label{sec-1}
%%%%%%%%%%%%%%%%%%%%%%%%%%%%%%%%%%%%%%%%%%%%  

Let $M$ be a compact 2-dimensional manifold, $G$ a compact connected 
semisimple Lie group acting as a gauge group, and $E$ a $G$-bundle over
$M$.  Motivated by 2-dimensional quantum gauge theories, Witten
\cite{Wi} evaluated the volume of the moduli space $\mathcal{M}$ of
flat connections on $E$ up to gauge transformations.  Such a result
can be regarded as a limit of Verlinde's formula \cite{Ve} when $M$ is
orientable, but Witten developed a more elementary method, based on
the decomposition of $M$ into three-holed spheres.  The result is now
called Witten's volume formula, which expresses the volume of
$\mathcal{M}$ in terms of special values of the Dirichlet series
\begin{equation}
  \label{0-1}
  \zeta_W (s;G)=\sum_{\psi}(\dim\psi)^{-s},
\end{equation}
where $\psi$ runs over all isomorphism classes of finite dimensional 
irreducible representations of $G$.

Let ${\frak g}={\rm Lie}(G)$ be the Lie algebra of $G$, and define
\begin{equation}
  \label{1-1}
  \zeta_W(s;{\frak g})=\sum_{\varphi}(\dim\varphi)^{-s},
\end{equation}
where the summation runs over all isomorphism classes of finite dimensional 
irreducible representations $\varphi$ of ${\frak g}$.
When $G$ is simply-connected, then 
there is a one-to-one correspondence between $\varphi$ and $\psi$.
In fact, each $\varphi$ is the differential of a certain
$\psi$, and so 
\begin{equation}\label{0-2}
  \zeta_W(s;{\frak g})=\zeta_W (s;G).
\end{equation}
Zagier \cite{Za} formulated the series \eqref{1-1} and called them Witten's
zeta-functions (see also Gunnells-Sczech \cite{GS}).   Witten's volume formula especially implies
\begin{equation}
  \label{1-2}
  \zeta_W(2k;{\frak g})=C_{W}(2k,{\frak g})\pi^{2kn}
\end{equation}
for $k\in{\Bbb N}$, where $n$ is the number of all positive roots of
${\frak g}$ and $C_{W}(2k,{\frak g})$ is a rational number.

Before proceeding further, here we fix several notations.  Let ${\Bbb
  N}$ be the set of positive integers, ${\Bbb N}_0={\Bbb N}\cup\{0\}$,
${\Bbb Z}$ the ring of rational integers, ${\Bbb Q}$ the rational
number field, ${\Bbb R}$ the real number field, ${\Bbb C}$ the complex
number field, respectively.

Let $\Delta$ be the set of all roots of ${\frak g}$, $\Delta_+$ the
set of all positive roots of ${\frak g}$ (hence $n=|\Delta_+|$),
$\fs=\{\alpha_1,\ldots,\alpha_r\}$ the fundamental system of $\Delta$,
$\alpha_j^{\vee}$ the coroot of $\alpha_j$ ($1\leq j\leq
r$).  Let $\lambda_1,\ldots,\lambda_r$ be the fundamental weights
satisfying
$\langle\alpha_i^{\vee},\lambda_j\rangle=\lambda_j(\alpha_i^{\vee}) =\delta_{ij}$ 
(Kronecker's delta).  

When $\mathfrak{g}=\mathfrak{sl}(2)$, the corresponding Witten zeta-function is
nothing but the Riemann zeta-function $\zeta(s)$ and \eqref{1-2} implies Euler's
well-known formula for $\zeta(2k)$.   Gunnells and Sczech \cite{GS} introduced a
method of computing $C_{W}(2k,{\frak g})$, and explicitly evaluated
$C_{W}(2k,{\frak{sl}(n)})$ for $n \geq 2$.

In \cite{KM2}, the authors introduced the multi-variable
version of Witten zeta-function
\begin{align}
  \label{1-3}
    \zeta_r({\bf s};{\frak g})=\sum_{m_1=1}^{\infty}\cdots
      \sum_{m_r=1}^{\infty}\prod_{\alpha\in\Delta_+}\langle\alpha^{\vee},
      m_1\lambda_1+\cdots+ m_r\lambda_r\rangle ^{-s_{\alpha}},
\end{align}
where ${\bf s}=(s_{\alpha})_{\alpha\in\Delta_+}\in {\Bbb C}^n$. 
When ${\frak g}$ is of type $X_r$, where $X=A,B,C,D,E,F,$ or $G$, we call 
\eqref{1-3} the
zeta-function of the root system of type $X_r$, and denote it by
$\zeta_r({\bf s};X_r)$.  Putting
\begin{align}
  \label{1-3-3}
   K({\frak g})=\prod_{\alpha\in\Delta_+}\langle\alpha^{\vee},
           \lambda_1+\cdots+\lambda_r\rangle,
\end{align}
and using \cite[(1.5) and (1.7)]{KM2}, we see that 
\begin{align}\label{wittenequalroot}
K({\frak g})^s\zeta_r(s,\ldots,s;{\frak g})=\zeta_W(s;{\frak g}).
\end{align}

In \cite{KM3}, the authors introduced a root-system theoretic
generalization of Bernoulli numbers and periodic Bernoulli functions,
and express $C_W(2k,\mathfrak{g})$ explicitly in terms of generalized 
periodic Bernoulli functions $\mathcal{P}(\mathbf{k},\mathbf{y};\Delta)$. 
Therefore we now have sufficiently explicit information on
formula \eqref{1-2}.
Moreover, in \cite{KMT,KMTJC,KM3,KM4,MT,TsC}, we proved various functional
relations among zeta-functions \eqref{1-1}, which include evaluation formulas
like \eqref{1-2} as special cases.
%Note that
%Szenes \cite{Sz98,Sz03} also studied generalizations of Bernoulli
%polynomials from the viewpoint of the theory of arrangement of
%hyperplanes, which include $\mathcal{P}(\mathbf{k},\mathbf{y};\Delta)$ mentioned above.

However the group $G$ is not necessarily simply-connected 
in Witten's paper \cite{Wi}.
(In fact, this point is emphasized by Witten himself in p.182 of \cite{Wi}.)
When $G$ is not simply-connected, relation \eqref{0-2} does not hold.
It is the aim of the present paper to consider such situation; that is, to study 
the zeta-functions and volume formulas in the sense of original formulation of 
Witten.

For this purpose, we introduce the multi-variable version of
$\zeta_W(s;G)$.   From \eqref{1-3} we see that the multi-variable version of
$\zeta_W(s;{\frak g})$ can be regarded as the zeta-function of 
the weight lattice of $\mathfrak{g}$.
%the lattice spanned by fundamental weights of $\mathfrak{g}$. %
   Similarly, in the present paper we will define 
a multi-variable
zeta-function of the weight lattice of $G$.
%a certain sublattice of the weight lattice, 
% which is associated 
% with $G$.
%In the present paper, we aim to extend our previous result in \cite{KM2,KM3} (see also \cite{KMTJC}) by considering zeta-functions of lattices associated with semisimple Lie groups. In fact, for a Lie group, the sublattice of the weight lattice can be uniquely determined, with which we can define the associated zeta-function. 
Actually this zeta-function, defined in Section \ref{sec-3}, is a partial sum 
of $\zeta_r({\bf s};{\frak g})$.
The volume formula for this zeta-function is given as Theorem \ref{thm:W-Z},
which gives an explicit formula for the values of this zeta-function at
${\bf s}=2{\bf k}$, where ${\bf k}=(k_{\alpha})_{\alpha\in\Delta_+}\in{\Bbb N}^n$
satisfying $k_{\alpha}=k_{\beta}$ if $\alpha$ and $\beta$ are of the same length.  
%just like the partial zeta-function. 
%It should be noted that Witten himself progressed his theory in this framework (see \cite{Wi}) (see Remark \ref{R-3-6}). 
%In the case that the lattice determined by a Lie group is the weight lattice, our present result coincides with the previous result in \cite{KM2,KM3}. 
As explicit examples, in
Section \ref{sec-4}, we consider the cases of $A_r$, $B_r$ and $C_r$ 
types ($r\leq 3$), and evaluate the associated zeta-functions in these cases. 

Since Theorem \ref{thm:W-Z} is a formula for ${\bf s}=2{\bf k}$, it is not useful
when we consider the values at odd integer points.   In order to study such cases,
in Sections \ref{sec-5}, we give some functional 
relations among zeta-functions of $A_2$ and $C_2(\simeq B_2)$ types. 
Those relations produce explicit formulas for special values of zeta-functions at
some points of the form 
${\bf s}={\bf l}=(l_{\alpha})_{\alpha\in\Delta_+}$,
where $l_{\alpha}\in{\Bbb N}$ and some of them are odd.
Those results include not only evaluation formulas given in Section \ref{sec-4} but also another type of evaluation formulas which can be regarded as certain extensions of the previous results in \cite{KMTJC,To,TsArch,TsC}. 
In Section \ref{sec-6}, we consider, what is called, \textit{parity results} 
for zeta values of $A_2$ and $C_2$ types.    We prove that parity results hold for
the zeta-functions associated with the groups $PU(3)$ and $PSp(2)$.

\

%%%%%%%%%%%%%%%%%%%%%%%%%%%%%%%%%%%%%%%%%%%
\section{A general form of zeta-functions}\label{sec-2}
%%%%%%%%%%%%%%%%%%%%%%%%%%%%%%%%%%%%%%%%%%%

We begin our theory with the definition of rather general form of zeta-functions.
We use the same notation as in \cite{KM5,KM2,KM3} 
(see also \cite{KMT,KMTpja,KMT-Mem,KM4}). 
For the details of basic facts about root systems and Weyl groups, see 
\cite{Bourbaki,Hum72,Hum}. 

Let $V$ be an $r$-dimensional real vector space equipped with an inner product $\langle \cdot,\cdot\rangle$.
The norm $\norm{\cdot}$ is defined by $\norm{v}=\langle v,v\rangle^{1/2}$.
The dual space $V^*$ is identified with $V$ via the inner product of $V$.
Let $\Delta$ be a finite reduced root system which may not be irreducible, and
$\fs=\{\alpha_1,\ldots,\alpha_r\}$ its fundamental system.
We fix 
$\Delta_+$ and $\Delta_-$ as the set of all positive roots and negative roots respectively.
Then we have a decomposition of the root system $\Delta=\Delta_+\coprod\Delta_-$ .
Let $Q=Q(\Delta)$ be the root lattice, $Q^\vee$ the coroot lattice,
$P=P(\Delta)$ the weight lattice, $P^\vee$ the coweight lattice,
$P_+$ the set of integral dominant weights 
and
$P_{++}$ the set of integral strongly dominant weights
respectively defined by
\begin{align}
&  Q=\bigoplus_{i=1}^r\mathbb{Z}\,\alpha_i,\qquad
  Q^\vee=\bigoplus_{i=1}^r\mathbb{Z}\,\alpha^\vee_i,\\
& P=\bigoplus_{i=1}^r\mathbb{Z}\,\lambda_i, \qquad 
  P^\vee=\bigoplus_{i=1}^r\mathbb{Z}\,\lambda^\vee_i,\\
&   P_+=\bigoplus_{i=1}^r\mathbb{N}_0\,\lambda_i, \qquad 
  P_{++}=\bigoplus_{i=1}^r\mathbb{N}\,\lambda_i,
\end{align}
where the fundamental weights $\{\lambda_j\}_{j=1}^r$
and
the fundamental coweights $\{\lambda_j^\vee\}_{j=1}^r$
are the dual bases of $\fs^\vee$ and $\fs$
satisfying $\langle \alpha_i^\vee,\lambda_j\rangle=\delta_{ij}$ and $\langle \lambda_i^\vee,\alpha_j\rangle=\delta_{ij}$ respectively.
A coweight $\mu\in P^\vee$ is said to be minuscule if 
$0\leq\langle\mu,\alpha\rangle\leq 1$ for all $\alpha\in\Delta$.
It is known that a minuscule coweight is one of fundamental coweights
and
as a system of representatives for $P^\vee/Q^\vee$,
we can take
$\{0\}\cup\{\lambda_j^\vee\}_{j\in J}$, where
$J$ is the set of all indices of minuscule coweights.

Let
\begin{equation}
\rho=\frac{1}{2}\sum_{\alpha\in\Delta_+}\alpha=\sum_{j=1}^r\lambda_j
\label{rho-def}
\end{equation}
be the lowest strongly dominant weight.
Then $P_{++}=P_++\rho$.
Let $\sigma_\alpha$ be the reflection with respect to a root $\alpha\in\Delta$ defined as 
\begin{equation}
 \sigma_\alpha:V\to V, \qquad \sigma_\alpha:v\mapsto v-\langle \alpha^\vee,v\rangle\alpha.
\end{equation}
For a subset $A\subset\Delta$, let
$W(A)$ be the group generated by reflections $\sigma_\alpha$ for all $\alpha\in A$. In particular, $W=W(\Delta)$ is the Weyl group, and 
$\{\sigma_j=\sigma_{\alpha_j}\,|\,1\leq j \leq r\}$ generates $W$. 
For $w\in W$, denote 
$\Delta_w=\Delta_+\cap w^{-1}\Delta_-$.

Let $\mathrm{Aut}(\Delta)$ be the subgroup of all the automorphisms $\GL(V)$ which 
stabilizes $\Delta$.
Then the Weyl group $W$ is a normal subgroup of $\mathrm{Aut}(\Delta)$ and
there exists a subgroup $\Omega\subset \mathrm{Aut}(\Delta)$ such that
$\mathrm{Aut}(\Delta)=\Omega\ltimes W$.                                                  
The subgroup $\Omega$ is isomorphic to the group $\mathrm{Aut}(\Gamma)$ of automorphisms
of the Dynkin diagram $\Gamma$
(see \cite[Section\,12.2]{Hum72}).

For a set $X$, 
denote by $\Func(X)$ the set of all complex valued functions on $X$.
For a function $f\in \Func(P)$, we define a subset
\begin{equation}
 H_f=\{\lambda\in P~|~f(\lambda)=0\}
\end{equation}
and for a subset $A$ of $\Func(P)$, define
$H_A=\bigcup_{f\in A}H_f$.
Note that an action of $W$ is induced on $\Func(P)$ 
as $(wf)(\lambda)=f(w^{-1}\lambda)$.

Let $f \in \Func(P/Q)$.
Since $P^\vee/Q^\vee$ is regarded as the dual of $P/Q$ over
$\mathbb{Q}/\mathbb{Z}$, $f$ can be expanded as follows:
\begin{equation}\label{2-expansion}
  f(\lambda)=\sum_{\mu\in P^\vee/Q^\vee}\widehat{f}(\mu)e^{2\pi i\langle\mu,\lambda\rangle},
\end{equation}
where $\langle\;,\;\rangle$ is regarded as an inner product on 
$P^\vee/Q^\vee$, and
$\widehat{f}:P^\vee/Q^\vee\to\mathbb{C}$ is given by
\begin{equation}\label{2-widehat_f_def}
  \widehat{f}(\mu)=\frac{1}{|P/Q|}\sum_{\lambda\in P/Q}f(\lambda)e^{-2\pi i\langle\mu,\lambda\rangle},
\end{equation}
because for $\nu\in P^\vee/Q^\vee$ we have
\begin{equation}
  \sum_{\lambda\in P/Q}e^{2\pi i\langle\nu,\lambda\rangle}=|P/Q|\delta_{\nu,0}.
\label{express-01}
%   \begin{cases}
%     |P/Q|\qquad&(\nu=0)\\
%     0\qquad&(\nu\neq 0).
%   \end{cases}
\end{equation}
Note that $f$ is automatically $W$-invariant
because for $\lambda\in P$
we have
$\sigma_\alpha(\lambda)=\lambda-\langle\alpha^\vee,\lambda\rangle\alpha\equiv
\lambda\pmod{Q}$.

For $\bs{s}=(s_{\alpha})\in \mathbb{C}^n$, 
$\bs{y} \in V$ and $f \in \Func(P/Q)$, we define
\begin{equation}
  \zeta_r(\bs{s},\bs{y},f;\Delta)
  =
  \sum_{\lambda\in P_{++}}
  f(\lambda)
  e^{2\pi i\langle\bs{y},\lambda\rangle}\prod_{\alpha\in\Delta_+}
  \frac{1}{\langle\alpha^\vee,\lambda\rangle^{s_\alpha}}.
\label{zeta-gene}
\end{equation}
Note that $\zeta_r(\bs{s},\bs{y},1;\Delta)$ was already studied in our previous 
work (see \cite[Section\,3]{KM5}, \cite[Section\,4]{KM3}). 
When $\Delta=\Delta(X_r)=\Delta({\frak g})$ is the root system 
attached to ${\frak g}$ of type $X_r$, then
$\zeta_r(\bs{s},\bs{0},1;\Delta)$ coincides with $\zeta_r(\bs{s};{\frak g})$
(see \eqref{1-3}).
For $w\in\mathrm{Aut}(\Delta)$, define the action of $w$ on $\zeta_r$ by
\begin{equation}
  (w\zeta_r)(\bs{s},\bs{y},f;\Delta)= 
  \zeta_r(w^{-1}\bs{s},w^{-1}\bs{y},w^{-1}f;\Delta),
\end{equation}
where $w^{-1}(\bs{s})=(s_{w\alpha})_{\alpha\in\Delta_+}$ (if 
$w\alpha\in\Delta_-$, we identify it with $-w\alpha$).   Then it is easy to see that for $w\in\mathrm{Aut}(\Gamma)$,
\begin{equation}
  (w\zeta_r)(\bs{s},\bs{y},f;\Delta)=\zeta_r(\bs{s},\bs{y},f;\Delta).
\end{equation}

\begin{remark} \label{Rem-y}
Why the exponential factor $e^{2\pi i\langle\bs{y},\lambda\rangle}$ is included in
the definition \eqref{zeta-gene}?    This is analogous to the Lerch zeta-function
\begin{equation}\label{Lerch-def}
\phi(s,\alpha)=\sum_{m=1}^{\infty}\frac{e^{2\pi im\alpha}}{m^s}.
\end{equation}
It is clear that the form \eqref{zeta-gene} with an exponential factor is useful
in the study of multiple series with twisting factors, such as 
the series discussed in Example \ref{Exam-A_2} and Remark \ref{Rem-SS}, or
multiple $L$-functions with Dirichlet characters \cite{KM5}.   
Moreover, using $\bs{y}$ we can simplify
our argument in various places.   In the argument below, \eqref{rel-1},
\eqref{gene-Ber} etc. are impossible to show without using $\bs{y}$.  
\end{remark}

\begin{remark} \label{Rem-y2}
It is to be noted that we may regard ${\bf y}\in V/Q^{\vee}$ in \eqref{zeta-gene}.
This is because when $a\in Q^{\vee}$ we have 
$\langle a,\lambda\rangle\in\mathbb{Z}$, hence
$e^{2\pi i\langle\bs{y}+a,\lambda\rangle}=e^{2\pi i\langle\bs{y},\lambda\rangle}$,
for any $\lambda\in P_{++}$.
\end{remark}

\begin{proposition}\label{Prop-conti}
The function $\zeta_r(\bs{s},\bs{y},f;\Delta)$, as a function in $\bs{s}$, can be
continued meromorphically to the whole space $\mathbb{C}^n$. 
\end{proposition}

\begin{proof}
By use of the expression \eqref{2-expansion}, with noting Remark \ref{Rem-y2},
we obtain
\begin{equation}
  \begin{split}
    \zeta_r(\bs{s},\bs{y},f;\Delta)
    &=
    \sum_{\lambda\in P_{++}}
    \sum_{\mu\in P^\vee/Q^\vee}\widehat{f}(\mu)e^{2\pi i\langle\mu,\lambda\rangle}
    e^{2\pi i\langle\bs{y},\lambda\rangle}\prod_{\alpha\in\Delta_+}
    \frac{1}{\langle\alpha^\vee,\lambda\rangle^{s_\alpha}}
    \\
    &=
    \sum_{\mu\in P^\vee/Q^\vee}\widehat{f}(\mu)
    \sum_{\lambda\in P_{++}}
    e^{2\pi i\langle\bs{y}+\mu,\lambda\rangle}\prod_{\alpha\in\Delta_+}
    \frac{1}{\langle\alpha^\vee,\lambda\rangle^{s_\alpha}}
    \\
    &=
    \sum_{\mu\in P^\vee/Q^\vee}\widehat{f}(\mu)
    \zeta_r(\bs{s},\bs{y}+\mu,1;\Delta).
  \end{split}
  \label{rel-1}
\end{equation}
In \cite[section 8]{KM5}, we showed that $\zeta_r(\bs{s},\mu,1;\Delta)$ can be 
continued meromorphically to the whole space, hence, so can be 
$\zeta_r(\bs{s},\bs{0},f;\Delta)$ from \eqref{rel-1}.
More generally the recent result of the first-named author in \cite{Komori09} gives 
that $\zeta_r(\bs{s},\bs{y},1;\Delta)$ $({\bs y}\in V)$ can
be continued meromorphically, so can be $\zeta_r(\bs{s},\bs{y},f;\Delta)$ from 
\eqref{rel-1}.
\end{proof}

Let 
\begin{equation}
  S(\bs{s},\bs{y},f;\Delta)=
  \sum_{\lambda\in P\setminus H_{\Delta^\vee}}f(\lambda)
  e^{2\pi i\langle\bs{y},\lambda\rangle}\prod_{\alpha\in\Delta_+}
    \frac{1}{\langle\alpha^\vee,\lambda\rangle^{s_\alpha}}.
\label{S-def-gen}    
\end{equation}    
Then, %by use of \eqref{Chi-1} and \eqref{Chi-2}, 
%similarly to the case of zeta-functions,
in the same way as in the case of zeta-functions,
we obtain
\begin{equation}
  S(\bs{s},\bs{y},f;\Delta)
  =
  \sum_{\mu\in P^\vee/Q^\vee}\widehat{f}(\mu)
  S(\bs{s},\bs{y}+\mu,1;\Delta).
\label{S-rel-gen}
\end{equation}
% \begin{equation}
%   \begin{split}
%     S(\bs{s},\bs{y},f;\Delta)
%     &=
%     \sum_{\lambda\in P\setminus H_{\Delta^\vee}}
%     f(\lambda)
%     e^{2\pi i\langle\bs{y},\lambda\rangle}\prod_{\alpha\in\Delta_+}
%     \frac{1}{\langle\alpha^\vee,\lambda\rangle^{s_\alpha}}
%     \\
%     &=
%     \sum_{\lambda\in P\setminus H_{\Delta^\vee}}
%     \sum_{\mu\in P^\vee/Q^\vee}\widehat{f}(\mu)e^{2\pi i\langle\mu,\lambda\rangle}
%     e^{2\pi i\langle\bs{y},\lambda\rangle}\prod_{\alpha\in\Delta_+}
%     \frac{1}{\langle\alpha^\vee,\lambda\rangle^{s_\alpha}}
%     \\
%     &=
%     \sum_{\mu\in P^\vee/Q^\vee}\widehat{f}(\mu)
%     \sum_{\lambda\in P\setminus H_{\Delta^\vee}}
%     e^{2\pi i\langle\bs{y}+\mu,\lambda\rangle}\prod_{\alpha\in\Delta_+}
%     \frac{1}{\langle\alpha^\vee,\lambda\rangle^{s_\alpha}}
%     \\
%     &=
%     \sum_{\mu\in P^\vee/Q^\vee}\widehat{f}(\mu)
%     S(\bs{s},\bs{y}+\mu,1;\Delta).
%   \end{split}
% \label{S-rel-gen}
% \end{equation}

Here we recall the generalized periodic Bernoulli functions 
$\mathcal{P}(\mathbf{k},\mathbf{y};\Delta)$ associated with $\Delta$ as follows 
(for the details, see \cite[Section 4]{KM3}). For 
$\mathbf{k}=(k_\alpha)_{\alpha\in\prs}\in \mathbb{N}_0^{n}$ and 
$\mathbf{y}\in V$ (or $\in V/Q^{\vee}$), 
we define
\begin{equation}
\label{eq:def_P}
  \begin{split}
  \mathcal{P}(\mathbf{k},\mathbf{y};\Delta)
&=
  \int_0^1\dots\int_0^1
  \Bigl(\prod_{\alpha\in\Delta_+\setminus\fs}
  B_{k_\alpha}(x_\alpha)\Bigr)
\\
&\qquad\times
  \biggl(
  \prod_{i=1}^r
  B_{k_{\alpha_i}}
  \Bigl(\Bigl\{
  \langle \mathbf{y},\lambda_i\rangle-\sum_{\alpha\in\Delta_+\setminus\fs}x_\alpha\langle\alpha^\vee,\lambda_i\rangle
  \Bigr\}\Bigr)
  \biggr)
  \prod_{\alpha\in\Delta_+\setminus\fs}
  dx_\alpha,
\end{split}
\end{equation}
where $\{ B_k(x)\}$ are the classical Bernoulli polynomials defined by
$$\frac{te^{xt}}{e^t-1}=\sum_{k=0}^\infty B_k(x)\frac{t^k}{k!}.$$ 
Then we have already obtained 
\begin{equation}
  \label{eq:S_B}
  S(\mathbf{k},\mathbf{y},1;\Delta)=
  (-1)^{n}
  \biggl(\prod_{\alpha\in\Delta_+}
  \frac{(2\pi\sqrt{-1})^{k_\alpha}}{k_\alpha!}\biggr)
  \mathcal{P}(\mathbf{k},\mathbf{y};\Delta)
\end{equation}
for $\mathbf{k}\in (\mathbb{N}_{\geq 2})^{n}$ 
(see \cite[(4.19)]{KM3}). 
%Note that 
%$\mathcal{P}(\mathbf{k},\mathbf{y}+a;\Delta)=
%\mathcal{P}(\mathbf{k},\mathbf{y};\Delta)$ for $a\in Q^\vee$
%and hence 
%$\mathcal{P}(\mathbf{k},\mathbf{y};\Delta)$
%is essentially a function on $V/Q^\vee$ in $\mathbf{y}$.
This function $\mathcal{P}(\mathbf{k},\mathbf{y};\Delta)$
may be regarded as a generalization of the periodic Bernoulli function 
and $\mathcal{B}_\mathbf{k}(\Delta)=\mathcal{P}(\mathbf{k},0;\Delta)$
the Bernoulli number.

Note that Szenes \cite{Sz98,Sz03} also studied generalizations of
Bernoulli polynomials from the viewpoint of the theory of arrangement
of hyperplanes, which include $\mathcal{P}(\mathbf{k},\mathbf{y};\Delta)$
mentioned above.

Suggested by \eqref{S-rel-gen}, we define generalized Bernoulli functions 
associated with $f$ and $\Delta$ by 
\begin{equation}
  \mathcal{P}(\mathbf{k},\mathbf{y},f;\Delta)=\sum_{\mu\in P^\vee/Q^\vee}\widehat{f}(\mu)
  \mathcal{P}(\mathbf{k},\mathbf{y}+\mu;\Delta). 
\label{gene-Ber}
\end{equation}
Then by \eqref{S-rel-gen}, \eqref{eq:S_B} and \eqref{gene-Ber}, we have
\begin{equation}
\label{eq:SfPf}
  S(\mathbf{k},\mathbf{y},f;\Delta)=
  (-1)^{n}
  \biggl(\prod_{\alpha\in\Delta_+}
  \frac{(2\pi\sqrt{-1})^{k_\alpha}}{k_\alpha!}\biggr)
  \mathcal{P}(\mathbf{k},\mathbf{y},f;\Delta)
\end{equation}
for $\mathbf{k}\in (\mathbb{N}_{\geq 2})^{n}$.

In \cite[Section 9]{KMT-Mem} and \cite[Section 3]{KM3}
we constructed the generating function of $\mathcal{P}(\mathbf{k},{\bf y};\Delta)$, which is 
\begin{equation}
F({\bf t},{\bf y};\Delta)=\sum_{{\bf k}\in \mathbb{N}_0^{n}} \mathcal{P}(\mathbf{k},{\bf y};\Delta)\prod_{\alpha\in \Delta_{+}}\frac{t^{k_\alpha}}{k_{\alpha}!}.
\label{F-gene}
\end{equation}
Since $F({\bf t},{\bf y};\Delta)$ can be evaluated explicitly 
(\cite[Theorem 4.1]{KM5}), we can evaluate $\mathcal{P}(\mathbf{k},\mathbf{y}+\mu;\Delta)$
from the expansion of $F({\bf t},{\bf y};\Delta)$.
In particular we find that $\mathcal{P}(\mathbf{k},\mu;\Delta)\in\mathbb{Q}$
for any $\mu\in P^\vee/Q^\vee$.
%Therefore we can compute $\mathcal{B}_\mathbf{k}(\mathcal{A};\Delta)$ explicitly by 
%\eqref{gene-Ber}.

\begin{theorem}
\label{thm:S-zeta}
For $\mathbf{s}=\mathbf{k}=(k_\alpha)_{\alpha\in\Delta_+}\in\mathbb{N}_{\geq2}^{n}$, $\bs{y}\in V$ and $f\in\mathfrak{F}(P/Q)$,
\begin{equation}
  \sum_{w\in W}
%  \Bigl(\prod_{\alpha\in\Delta_+\cap w^{-1}\Delta_-}
  \Bigl(\prod_{\alpha\in\Delta_+\cap w\Delta_-}
  (-1)^{k_{\alpha}}\Bigr)
%  \zeta_r(w^{-1}\mathbf{k},w^{-1}\mathbf{y},w^{-1}f;\Delta)
  \zeta_r(w^{-1}\mathbf{k},w^{-1}\mathbf{y},f;\Delta)
  \\
%  =(-1)^{\abs{\Delta_+}}  
  =(-1)^{n}
  \biggl(\prod_{\alpha\in\Delta_+}
  \frac{(2\pi\sqrt{-1})^{k_\alpha}}{k_\alpha!}\biggr)  \mathcal{P}(\mathbf{k},\mathbf{y},f;\Delta).
\end{equation}
\end{theorem}

\begin{proof}
Since $P\setminus H_{\Delta^\vee}=\bigcup_{w\in W}w(P_{++})$, we have
\begin{equation}
  \begin{split}
    S(\bs{s},\bs{y},f;\Delta)
    &=
    \sum_{\lambda\in P\setminus H_{\Delta^\vee}}f(\lambda)
    e^{2\pi i\langle\bs{y},\lambda\rangle}\prod_{\alpha\in\Delta_+}
    \frac{1}{\langle\alpha^\vee,\lambda\rangle^{s_\alpha}}
    \\
    &=
    \sum_{w\in W}
    \sum_{\lambda\in P_{++}}f(w\lambda)
    e^{2\pi i\langle\bs{y},w\lambda\rangle}\prod_{\alpha\in\Delta_+}
    \frac{1}{\langle\alpha^\vee,w\lambda\rangle^{s_\alpha}}
    \\
    &=
    \sum_{w\in W}
    \sum_{\lambda\in P_{++}}(w^{-1}f)(\lambda)
    e^{2\pi i\langle w^{-1}\bs{y},\lambda\rangle}\prod_{\alpha\in\Delta_+}
    \frac{1}{\langle w^{-1}\alpha^\vee,\lambda\rangle^{s_\alpha}}
    \\
    &=
    \sum_{w\in W}
    \Bigl(\prod_{\alpha\in\Delta_{w^{-1}}}(-1)^{-s_\alpha}\Bigr)
    \zeta_r(w^{-1}\bs{s},w^{-1}\bs{y},w^{-1}f;\Delta),
  \end{split}
\end{equation}
where the last equality follows by rewriting $\alpha$ to $w\alpha$, and when
$\alpha\in -\Delta_w=w^{-1}\Delta_+\cap\Delta_-$ further replacing $\alpha$ by
$-\alpha$ (see the proof of Theorem 4.3 in \cite{KM3}).

Combining \eqref{eq:SfPf}
% Combining \eqref{S-rel-gen},
% \eqref{eq:S_B} and \eqref{gene-Ber}, 
and the $W$-invariance of $f$,
 we have the result.
\end{proof}

In the following sections, we treat some special cases.
Let $\mathcal{A}\subset P$ with $\mathcal{A}+Q=\mathcal{A}$.
Let
$\iota_\mathcal{A}:P\to \{0,1\}$ be the characteristic function of $\mathcal{A}$ defined by
\begin{equation}
  \iota_\mathcal{A}(\lambda)=
  \begin{cases}
    1\qquad &(\lambda\in \mathcal{A}),\\
    0\qquad &(\lambda\not\in \mathcal{A}).
  \end{cases}
\end{equation}
Then $\iota_\mathcal{A}$ can be regarded as a function on $P/Q$.
Hence \eqref{2-expansion} and \eqref{2-widehat_f_def} implies that
\begin{equation}
  \iota_\mathcal{A}(\lambda)=\sum_{\mu\in P^\vee/Q^\vee}\widehat{\iota_\mathcal{A}}(\mu)e^{2\pi i\langle\mu,\lambda\rangle},
\label{Chi-1}
\end{equation}
where $\widehat{\iota_\mathcal{A}}:P^\vee/Q^\vee\to\mathbb{C}$ is given by
\begin{equation}
  \widehat{\iota_\mathcal{A}}(\mu)=\frac{1}{|P/Q|}\sum_{\lambda\in P/Q}
\iota_\mathcal{A}(\lambda)e^{-2\pi i\langle\mu,\lambda\rangle}
=\frac{1}{|P/Q|}\sum_{\lambda\in \mathcal{A}/Q}e^{-2\pi i\langle\mu,\lambda\rangle}.
\label{Chi-2}
\end{equation}

\section{Zeta-functions of weight lattices of Lie groups} \label{sec-3}
%%%%%%%%%%%%%%%%%%%%%%%%%%%%                                                            

Now we define zeta-functions of weight lattices of Lie groups. 
Let $\widetilde{G}$ be a simply-connected compact semisimple Lie group,
and ${\frak g}={\rm Lie}(\widetilde{G})$.
There is a one-to-one correspondence between 
a connected compact semisimple Lie group $G$ whose universal covering group is
$\widetilde{G}$, and a lattice $L$ with 
$Q(\Delta({\frak g}))\subset L\subset P(\Delta({\frak g}))$ up to 
automorphisms (see Remark \ref{rem-one-to-one})
by taking  $L=L(G)$ as the weight lattice of $G$.
Let $L_+=P_+\cap L$. 

We define the zeta-function of the weight lattice $L=L(G)$ of the semisimple Lie group $G$ by 
\begin{equation}
  \zeta_r(\bs{s},\bs{y};G)=
  \zeta_r(\bs{s},\bs{y};L;\Delta):=\sum_{\lambda\in L_++\rho}
    e^{2\pi i\langle\bs{y},\lambda\rangle}\prod_{\alpha\in\Delta_+}
    \frac{1}{\langle\alpha^\vee,\lambda\rangle^{s_\alpha}}. 
\label{Lerch-zeta-L}
\end{equation}
This is the case $f=\iota_{\mathcal{A}}$, $\mathcal{A}=L+\rho$ 
and $\Delta$ of \eqref{zeta-gene},
and so, 
by Proposition \ref{Prop-conti}, we see that this zeta-function can be continued 
meromorphically to $\mathbb{C}^{n}$.
When ${\bf y}={\bf 0}$, we sometimes write this zeta-function as $\zeta_r({\bf s};G)$ or
$\zeta_r(\bs{s};L;\Delta)$ for brevity.
It is to be noted that if $G=\widetilde{G}$, then $L=P$ and 
$\zeta_r({\bf s};\widetilde{G})$ coincides with $\zeta_r({\bf s};{\frak g})$
defined in Section \ref{sec-1}.

% Using Lemma \ref{lm:key2}, we obtain
% \begin{equation}
%   \begin{split}
%     S(\bs{s},\bs{y};L+\rho;\Delta)
%     &=
%     \sum_{\lambda\in (L+\rho)\setminus H_{\Delta^\vee}}
%     e^{2\pi i\langle\bs{y},\lambda\rangle}\prod_{\alpha\in\Delta_+}
%     \frac{1}{\langle\alpha^\vee,\lambda\rangle^{s_\alpha}}
%     \\
%     &=
%     \sum_{w\in W}
%     \sum_{\lambda\in L_++\rho}
%     e^{2\pi i\langle\bs{y},w\lambda\rangle}\prod_{\alpha\in\Delta_+}
%     \frac{1}{\langle\alpha^\vee,w\lambda\rangle^{s_\alpha}}
%     \\
%     &=
%     \sum_{w\in W}
%     \sum_{\lambda\in L_++\rho}
%     e^{2\pi i\langle w^{-1}\bs{y},\lambda\rangle}\prod_{\alpha\in\Delta_+}
%     \frac{1}{\langle w^{-1}\alpha^\vee,\lambda\rangle^{s_\alpha}}
%     \\
%     &=
%     \sum_{w\in W}
%     \Bigl(\prod_{\alpha\in\Delta_{w^{-1}}}(-1)^{-s_\alpha}\Bigr)
%     \zeta_r(w^{-1}\bs{s},w^{-1}\bs{y};L;\Delta),
%   \end{split}
% \label{S-zeta}
% \end{equation}
% where the last equality follows by rewriting $\alpha$ to $w\alpha$, and when
% $\alpha\in -\Delta_w=w^{-1}\Delta_+\cap\Delta_-$ further replacing $\alpha$ by
% $-\alpha$ (see the proof of Theorem 4.3 in \cite{KM3}).

For any lattice $M$, we define $M^*=\Hom(M,\mathbb{Z})$.
Since $Q^*=P^\vee$ and $P^*=Q^\vee$,
from $Q\subset L\subset P$ we obtain
\begin{align}\label{totemoatui}
 P^{\vee}=Q^*\supset L^*\supset P^*=Q^{\vee}.
\end{align}
% This induces the natural projection 
% $\psi:
% P^\vee/Q^\vee=Q^*/P^*
% \to 
% Q^*/L^*%=P^\vee/L^*
% $.
We define
\begin{equation}
  \delta_{L^*/Q^\vee}(\mu)=
  \Biggl\{
  \begin{aligned}
    1\qquad&(\mu\in L^*/Q^\vee),\\
    0\qquad&(\mu\notin L^*/Q^\vee).
  \end{aligned}
\end{equation}

\begin{proposition} \label{P-3-4}
Let $L$ be a lattice satisfying $Q\subset L\subset P$.
For $\mu\in P^\vee/Q^\vee$, we have
\begin{equation}
  \widehat{\iota_{L+\rho}}(\mu)=
  \frac{(-1)^{\langle\mu,2\rho\rangle}}{|P/L|}\delta_{L^*/Q^\vee}(\mu)%\delta_{\mu\in L^*} %\delta_{\psi(\mu),0}
  \in\mathbb{Q}.
\label{chi-wide-L}
\end{equation}  
\end{proposition}

\begin{proof}
We have
\begin{equation*}
  \sum_{\lambda\in (L+\rho)/Q}e^{-2\pi i\langle\mu,\lambda\rangle}
%   =
%   e^{-2\pi i\langle\mu,\rho\rangle}
%   \sum_{\lambda\in L/Q}e^{-2\pi i\langle\mu,\lambda\rangle}
  =
  (-1)^{\langle\mu,2\rho\rangle}
  \sum_{\lambda\in L/Q}e^{-2\pi i\langle\mu,\lambda\rangle}.
\end{equation*}
Note that $(-1)^{\langle\mu,2\rho\rangle}\in\{1,-1\}$
due to $\rho\in Q/2$.
We obtain
\begin{equation*}
  \sum_{\lambda\in L/Q}e^{-2\pi i\langle\mu,\lambda\rangle}=
  \begin{cases}
    |L/Q| \qquad &(\mu\in L^*/Q^\vee),\\
    0\qquad &(\mu\notin L^*/Q^\vee).
  \end{cases}
\end{equation*}
Therefore %\eqref{2-widehat_f_def} 
\eqref{Chi-2}
gives
\begin{equation*}
  \widehat{\iota_{L+\rho}}(\mu)=(-1)^{\langle\mu,2\rho\rangle}
\frac{|L/Q|}{|P/Q|}\delta_{L^*/Q^\vee}(\mu).%\delta_{\mu\in L^*}.
\end{equation*}
\end{proof}

In particular
\begin{align}
  \widehat{\iota_{P+\rho}}(\mu)
  &=\frac{1}{|P/Q|}\sum_{\lambda\in (P+\rho)/Q}e^{-2\pi i\langle\mu,\lambda\rangle}
  =\delta_{\mu,0}
\end{align}
and
\begin{align}
  \widehat{\iota_{Q+\rho}}(\mu)
  &=\frac{1}{|P/Q|}\sum_{\lambda\in (Q+\rho)/Q}e^{-2\pi i\langle\mu,\lambda\rangle}
  =\frac{(-1)^{\langle\mu,2\rho\rangle}}{|P/Q|}.
\end{align}

We define
\begin{equation}
  \mathcal{P}(\mathbf{k},\mathbf{y};L;\Delta)
  =
  \mathcal{P}(\mathbf{k},\mathbf{y},\iota_{L+\rho};\Delta).
\end{equation}
% \begin{equation}
%   \begin{split}
%     \mathcal{P}(\mathbf{k},\mathbf{y};L;\Delta)
%     &=
%     \mathcal{P}(\mathbf{k},\mathbf{y},\iota_{L+\rho};\Delta)
%     \\
%     &=
%     \frac{1}{|P/L|}
%     \sum_{\mu\in L^*/Q^\vee}    
%     (-1)^{\langle\mu,2\rho\rangle}
%     \mathcal{P}(\mathbf{k},\mathbf{y}+\mu;\Delta). 
%   \end{split}
% \end{equation}
Note that
since $P/L\simeq L^*/Q^{\vee}\simeq\pi_1(G)$, this can also be written as
% we have by \eqref{gene-Ber},
\begin{equation}
  \mathcal{P}(\mathbf{k},\mathbf{y};L;\Delta)
  =
  \frac{1}{|\pi_1(G)|}
  \sum_{\mu\in \pi_1(G)}    
  (-1)^{\langle\mu,2\rho\rangle}
  \mathcal{P}(\mathbf{k},\mathbf{y}+\mu;\Delta)
\end{equation}
by            
\eqref{gene-Ber} and Proposition \ref{P-3-4}. 

We can compute $\mathcal{P}(\mathbf{k},\mathbf{y};L;\Delta)$ explicitly by            
\eqref{gene-Ber}. 
In particular, combining with \eqref{chi-wide-L} we have for $\nu\in P^\vee/Q^\vee$,
\begin{equation}
  \mathcal{P}(\mathbf{k},\nu;L;\Delta) \in \mathbb{Q}.
  \label{Ber-Q}
\end{equation}
From this fact we can deduce the following.
%Therefore, combining Lemma \ref{lemlem} and \eqref{S-zeta}, we obtain the following.

\begin{theorem}\label{thm:W-Z} 
  For a compact connected semisimple Lie group $G$, let
  $\Delta=\Delta(G)$ be its root system, 
  and $L=L(G)$ be its weight lattice.
  Let
  $\mathbf{k}=(k_\alpha)_{\alpha\in\prs}\in\mathbb{N}^{n}$ {\rm
    ($n=|\Delta_+|$)} satisfying $k_\alpha=k_\beta$ whenever
  $\norm{\alpha}=\norm{\beta}$.  Let
  $\kappa=\sum_{\alpha\in\Delta_+}{2k_\alpha}$.  Then we have
% \begin{equation}
%   \begin{split}
%     \zeta_r& (2\mathbf{k},\bs{y};G) =\zeta_r(2\mathbf{k},\mathbf{y};L;\Delta)\\
%     & =\frac{(-1)^{n}}{\abs{W}}
%     \biggl(\prod_{\alpha\in\Delta_+}
%     \frac{(2\pi\sqrt{-1})^{2k_\alpha}}{(2k_\alpha)!}\biggr)
%     \mathcal{P}(2\mathbf{k},\mathbf{y};L+\rho;\Delta),
%   \end{split}
% \end{equation}
% and in particular, 
for $\nu\in P^\vee/Q^\vee$,
\begin{equation}
  \label{eq:vol_formula}
  \begin{split}
    \zeta_r& (2\mathbf{k},\nu;G) =\zeta_r(2\mathbf{k},\nu;L;\Delta)\\
    & =\frac{(-1)^{n}}{\abs{W}}
    \biggl(\prod_{\alpha\in\Delta_+}
    \frac{(2\pi\sqrt{-1})^{2k_\alpha}}{(2k_\alpha)!}\biggr)
    \mathcal{P}(2\mathbf{k},\nu;L;\Delta) \in \mathbb{Q}\cdot 
    \pi^{\kappa}.
  \end{split}
\end{equation}
%where $\rho$ is the lowest strongly dominant weight {\rm (}see \eqref{rho-def}{\rm)}.
\end{theorem}

\begin{proof}
By Theorem \ref{thm:S-zeta} with
 ${\bf s}=2{\bf k}$ and ${\bf y}=\nu$,
%Putting ${\bf s}=2{\bf k}$ and ${\bf y}=\nu$ in \eqref{S-zeta}, 
we obtain
\begin{equation*}
%  S(2{\bf k},\nu,\iota_{L+\rho};\Delta)=
\sum_{w\in W}\zeta_r(w^{-1}(2{\bf k}),w^{-1}\nu;L;\Delta)
=
(-1)^n 
  \biggl(\prod_{\alpha\in\Delta_+}
  \frac{(2\pi\sqrt{-1})^{k_\alpha}}{k_\alpha!}\biggr)\mathcal{P}(\bs{k},\bs{y};L;\Delta).
\end{equation*}
Since roots of the same length form a single Weyl-orbit
and $w^{-1}\nu\equiv \nu\pmod{Q^\vee}$, the left-hand side
of the above is
$$
\sum_{w\in W}\zeta_r(2{\bf k},\nu;L;\Delta)
=|W|\zeta_r(2{\bf k},\nu;L;\Delta).
$$
The assertion follows from this %, Lemma \ref{lemlem} 
and \eqref{Ber-Q}.
\end{proof}

This theorem is the explicit form of the volume formula for the zeta-function
of the lattice $L=L(G)$.
In the case when $L=P$, \eqref{eq:vol_formula} coincides with our previous result 
in \cite[Theorem 4.6]{KM3}. 
%Hence, in this sense, we can regard \eqref{eq:vol_formula} as an extended Witten volume formula for the zeta-function of the lattice $L$. Actually, Witten himself considered this type of volume formulas for lattices associated with a Lie group $G$ in \cite{Wi}. Therefore Theorem \ref{thm:W-Z} implies a certain formulation of Witten's original work. 

\begin{remark}\label{rem-one-to-one}
The correspondence between Lie groups and lattices is a well-known fact, but here
we sketch the demonstration for the convenience of readers.   
Let $G$ be a compact connected semisimple Lie group, whose universal covering
group is $\widetilde{G}$.   There is a one-to-one correspondence between 
isomorphism classes of finite dimensional irreducible representations of $G$ and
dominant analytically integral forms (for example, \cite[Theorem 5.110]{Kn}).
The set of dominant analytically integral forms produces a sublattice $L=L(G)$
of the weight lattice (or the lattice of algebraically integral forms) $P$ of $\mathfrak{g}$, and
$L$ includes the root lattice $Q$ (\cite[(4.63)]{Kn}).   
In particular $L(\widetilde{G})=P$.
Conversely, let $\widetilde{G}$ be a simply-connected Lie group,
${\frak g}={\rm Lie}(\widetilde{G})$, and let $L$ be a lattice satisfying
$Q=Q(\Delta({\frak g}))\subset L\subset P=P(\Delta({\frak g}))$.   
Then \eqref{totemoatui} holds.
Since $P^{\vee}/Q^{\vee}$ is isomorphic to the center $\widetilde{Z}$ of 
$\widetilde{G}$ by the mapping
$$
\Phi\;:\;P^{\vee}\ni\mu\mapsto\exp_{\widetilde{G}}(2\pi i\mu)\in\widetilde{Z}
$$
(where $\exp_{\widetilde{G}}$ means the exponential mapping associated with
$\widetilde{G}$), 
we may regard $L^*/Q^{\vee}$ as a subgroup of $\widetilde{G}$.
Define $G=\widetilde{G}/(L^*/Q^{\vee})$.   
We show that the lattice corresponding to $G$ is $L$.   Write $L_1=L(G)$.
Take a maximal torus $K$ of $G$, and ${\frak k}={\rm Lie}(G)$.
Let $\lambda\in L$ and $H\in{\frak k}$ with $\exp_{G}(H)=1$.   The last
condition implies $\exp_{\widetilde{G}}(H)\in\Phi(L^*)$, so
$H\in 2\pi i L^*$, $\lambda(H)\in 2\pi i\mathbb{Z}$.
Therefore $\lambda\in L_1$ by \cite[Proposition 4.58]{Kn},
hence $L\subset L_1$.   
On the other hand $(L_1^*:Q^{\vee})=(L_1^*:P^*)=(P:L_1)$, but the right-hand side
is equal to $(L^*:Q^{\vee})$ by \cite[Proposition 4.67]{Kn}.   Therefore
$(L_1^*:Q^{\vee})=(L^*:Q^{\vee})$, and we can
conclude that $L_1=L$.
\end{remark}

\ 

%%%%%%%%%%%%%%%%%%%%%%%%%%%%%%%%%
\section{Explicit forms of zeta-functions} \label{sec-4}
%%%%%%%%%%%%%%%%%%%%%%%%%%%%%%%%%

In this section, we will give several examples of explicit forms of zeta-functions 
defined by \eqref{Lerch-zeta-L}.
When $L=P$, the zeta-function is nothing but $\zeta_r({\bf s};{\frak g})$,
so our main concern is the case $P\supsetneq L\supset Q$.

\begin{example} \label{Exam-A_2} 
We first study the case of $A_2$ type. 
Let $\Delta=\Delta(A_2)$ with $\Psi=\{ \alpha_1,\,\alpha_2\}$, 
$\Delta_{+}=\{ \alpha_1,\,\alpha_2,\,\alpha_1+\alpha_2\}$, 
$P=\mathbb{Z}\lambda_1+\mathbb{Z}\lambda_2$, 
$Q=\mathbb{Z}\alpha_1+\mathbb{Z}\alpha_2$, and $\rho=\lambda_1+\lambda_2$.
It is known that $(P:Q)=3$ (see \cite[Planche]{Bourbaki}). 
Therefore the only lattice $L$ with $P\supsetneq L \supset Q$ is $Q$.  
Then $Q_{+}=P_{+}\cap Q$.
%  and
% \begin{equation}
% Q_{+}+\rho=\left( (Q+\rho)\setminus H_{\Delta^{\vee}}\right)\cap P_{+}
% \label{Q1}
% \end{equation}
% (Lemma \ref{lm:key1}).   
We show
\begin{equation}
Q_{+}+\rho=\left\{ m_1\lambda_1+m_2\lambda_2\,|\,m_1,m_2 \in \mathbb{N},
\ m_1 \equiv m_2\ (\text{mod}\ 3)\right\}.
 \label{Q2}
\end{equation}
To show this, first note that 
\begin{align}\label{lambda-alpha}
\lambda_1=\frac{2}{3}\alpha_1+\frac{1}{3}\alpha_2,\quad
\lambda_2=\frac{1}{3}\alpha_1+\frac{2}{3}\alpha_2.
\end{align}
In fact, write 
$\lambda_1=u\alpha_1+v\alpha_2$.   Since
$\alpha_i^\vee=2\alpha_i/\langle \alpha_i,\alpha_i \rangle$ and 
$\langle \alpha_i^\vee,\lambda_j\rangle=\delta_{ij}$, 
we have
$$
1=\langle\alpha_1^{\vee},\lambda_1\rangle
=\frac{2}{\langle \alpha_1,\alpha_1 \rangle}(u\langle \alpha_1,\alpha_1 \rangle
+v\langle \alpha_1,\alpha_2 \rangle)=2u-v,
$$
and $0=\langle\alpha_2^{\vee},\lambda_1\rangle=-u+2v$, from which we obtain
$u=2/3$, $v=1/3$, so $\lambda_1=(2/3)\alpha_1+(1/3)\alpha_2$.   The case of
$\lambda_2$ is similar.

Let $\lambda=m_1\lambda_1+m_2\lambda_2\in Q_{+}+\rho$ ($m_1,m_2\in\mathbb{N}$).
Then $\lambda-\rho=n_1\lambda_1+n_2\lambda_2\in Q_{+}$, where $n_j=m_j-1$ ($j=1,2$).
From \eqref{lambda-alpha} we have
$$
n_1\lambda_1+n_2\lambda_2=\frac{2n_1+n_2}{3}\alpha_1+\frac{n_1+2n_2}{3}\alpha_2.
$$
Since this belongs to $Q$, we have $(2n_1+n_2)/3\in\mathbb{Z}$ and
$(n_1+2n_2)/3\in\mathbb{Z}$, which are equivalent to 
$n_1 \equiv n_2\ (\text{mod}\ 3)$.                                                      
This implies \eqref{Q2}.

The simply-connected group $\widetilde{G}$ in the $A_2$ case is $SU(3)$. 
Let $\widetilde{Z}$ be the center of $\widetilde{G}$.
The group corresponding to $Q$ is
% $\widetilde{G}/\widetilde{Z}$, 
 $\widetilde{G}/\widetilde{Z}$, 
which is the 
projective unitary group $PU(3)$.
The zeta-function corresponding to $P$ is 
\begin{align}\label{MT}
\zeta_2((s_1,s_2,s_3),\bs{y};SU(3))=                                                   
  \zeta_2 ((s_1,s_2,s_3),\bs{y};P;A_2)=
  \sum_{m,n=1}^\infty                                 
  \frac{e^{2\pi i\langle\bs{y},m\lambda_1+n\lambda_2\rangle}}{m^{s_1}n^{s_2}(m+n)^{s_3}},
\end{align}
which, when $\bs{y}={\bf 0}$, is the classical Mordell-Tornheim double sum.
In the case 
$\bs{y}=\lambda_1^\vee=\frac{2}{3}\alpha_1^\vee+\frac{1}{3}\alpha_2^\vee$, we have
\begin{equation}
\zeta_2((s_1,s_2,s_3),\lambda_1^\vee;SU(3))=  \sum_{m,n=1}^\infty 
  \frac{\varrho^{2m+n}}{m^{s_1}n^{s_2}(m+n)^{s_3}},
\label{A2-varrho}
\end{equation}
where $\varrho=e^{2\pi i/3}$ is the cube root of unity. 
The zeta-function corresponding to $Q$ is,
by \eqref{Lerch-zeta-L} and \eqref{Q2}, 
\begin{equation}
\begin{split}
  &\zeta_2((s_1,s_2,s_3),\bs{y};PU(3))=
  \zeta_2 ((s_1,s_2,s_3),\bs{y};Q;A_2)\\ 
    &=\sum_{\lambda\in Q_++\rho}\frac{e^{2\pi i\langle\bs{y},\lambda\rangle}}
    {\langle\alpha_1^\vee,\lambda\rangle^{s_1}\langle\alpha_2^\vee,\lambda\rangle^{s_2}\langle \alpha_1^\vee+\alpha_2^\vee,\lambda\rangle^{s_3}} 
    =\sum_{m,n=1 \atop m\equiv n\,(\text{mod}\,3)}^\infty 
     \frac{e^{2\pi i\langle\bs{y},m\lambda_1+n\lambda_2\rangle}}{m^{s_1}n^{s_2}(m+n)^{s_3}}.
\end{split}
\label{Lerch-zeta-A2}
\end{equation}
This can be regarded as a kind of ``partial zeta-function'' of $A_2$ type, the
double analogue of the partial (Riemann) zeta-function.

For $\mathbf{y}=y_1\alpha_1^\vee+y_2\alpha_2^\vee$, we can compute $\mathcal{P}(\mathbf{k},\mathbf{y};A_2)$ from their generating function given by 
\begin{equation}
\begin{split}
    F(\mathbf{t},\mathbf{y};A_2)&=
\frac{t_3}{t_3-t_1-t_2}\frac{t_1e^{t_1 \{y_1\}}}{e^{t_1}-1}\frac{t_2e^{t_2 \{y_2\}}}{e^{t_2}-1}\\
&\quad+    \frac{t_2}{t_2+t_1-t_3}\frac{t_1e^{t_1 \{y_1-y_2\}}}{e^{t_1}-1}\frac{t_3e^{t_3\{y_2\}}}{e^{t_3}-1}
\\
&\quad+    \frac{t_1}{t_1+t_2-t_3}\frac{t_2e^{t_2(1-\{y_1-y_2\})}}{e^{t_2}-1}\frac{t_3e^{t_3\{y_1\}}}{e^{t_3}-1},
\end{split}
\end{equation}
where $\{\cdot\}$ denotes the fractional part (see \cite[Section 9]{KMT-Mem}). 
For example, we have
\begin{multline} 
  \mathcal{P}((2,2,2),\bs{y};A_2)
  % 0
  =\frac{1}{3780}
  % 1
  +\frac{1}{90}(\{y_1\}-\{y_1-y_2\}-\{y_2\})
  \\
  \begin{aligned}
    % &
    % \cdots\\
    % 2
    &
    +\frac{1}{90}(
    -\{y_1\}^2
    -2 \{y_1-y_2\}\{y_1\}
    + \{y_1-y_2\}^2
    -\{y_2\}^2
    +2 \{y_1-y_2\} \{y_2\})
    \\
    % 3
    &
    +\frac{1}{18}(
    -\{y_1\}^3
    +3 \{y_1-y_2\} \{y_1\}^2
    +3 \{y_2\}^3
    +3 \{y_1-y_2\} \{y_2\}^2
    )
    \\
    % 4
    &
    +\frac{1}{18}(
    \{y_1\}^4
    -2 \{y_1-y_2\} \{y_1\}^3
    -3 \{y_1-y_2\}^2 \{y_1\}^2
    \\
    &\qquad\qquad
    -5 \{y_2\}^4
    -10 \{y_1-y_2\}\{y_2\}^3
    -3 \{y_1-y_2\}^2 \{y_2\}^2
    )
    \\
    % 5
    &
    +\frac{1}{30}(
    \{y_1\}^5
    -5 \{y_1-y_2\} \{y_1\}^4
    +10 \{y_1-y_2\}^2 \{y_1\}^3
    \\
    &\qquad\qquad
    +5 \{y_2\}^5
    +15 \{y_1-y_2\} \{y_2\}^4
    +10 \{y_1-y_2\}^2 \{y_2\}^3
    )
    \\
    % 6
    &
    +\frac{1}{30}(
    -\{y_1\}^6
    +4 \{y_1-y_2\} \{y_1\}^5
    -5 \{y_1-y_2\}^2 \{y_1\}^4
    \\
    &
    \qquad\qquad-\{y_2\}^6
    -4 \{y_1-y_2\} \{y_2\}^5
    -5 \{y_1-y_2\}^2 \{y_2\}^4
    )
  \end{aligned}
\label{Ber-A2-gen}
\end{multline}
(see \cite[(9.13)]{KMT-Mem}).  Hence, when $\bf{y}=\bf{0}$, it follows from 
\eqref{gene-Ber} and \eqref{Ber-A2-gen} that
%$B_{2\mathbf{k}}(Q+\rho;A_2)=\frac{187}{2755620}$. 
%$P(2\mathbf{k},\mathbf{0};Q+\rho;A_2)=\frac{187}{2755620}$. 
$$\mathcal{P}((2,2,2),\mathbf{0};Q;A_2)=\frac{187}{2755620}.$$
Therefore we obtain
from Theorem \ref{thm:W-Z} that
\begin{equation}
\begin{split}
  \zeta_2 ((2,2,2);PU(3)) 
    &=\sum_{m,n=1 \atop m\equiv n\,(\text{mod}\,3)}^\infty 
     \frac{1}{m^{2}n^{2}(m+n)^{2}}=\frac{187}{688905}\pi^6.
\end{split}
\label{zeta-A2-val}
\end{equation}
Similarly we can compute 
\begin{align}
& \zeta_2 ((4,4,4);PU(3)) =\frac{3279473}{48475988686125} \pi^{12}, \label{zeta-A2-4}\\
& \zeta_2 ((6,6,6);PU(3)) =\frac{53109402098}{3020275543157103456225} \pi^{18}, \label{zeta-A2-6}\\
& \zeta_2 ((8,8,8);PU(3)) =\frac{178778564412743}{39097800024794787744890296875} \pi^{24}. \label{zeta-A2-8}
\end{align}
Also, in the case $\textbf{y}=\lambda_1^\vee=\frac{2}{3}\alpha_1^\vee+\frac{1}{3}\alpha_2^\vee$ (see \eqref{A2-varrho}), that is, $(y_1,y_2)=\left(\frac{2}{3},\frac{1}{3}\right)$, we can similarly obtain
\begin{align}
& \zeta_2((2,2,2),\lambda_1^\vee;SU(3))=\frac{53}{229635} \pi^6 \label{su3-2}\\
& \zeta_2((4,4,4),\lambda_1^\vee;SU(3))=\frac{1078771}{16158662895375} \pi^{12} \label{su3-4}\\
& \zeta_2((6,6,6),\lambda_1^\vee;SU(3))=\frac{88392335894}{5033792571928505760375}\pi^{18}\label{su3-6}\\
& \zeta_2((8,8,8),\lambda_1^\vee;SU(3))=\frac{1012923518531597}{221554200140503797221045015625}\pi^{24}.\label{su3-8}
\end{align}
Note that from the definition, we can confirm
\begin{align*}
\zeta_2((2p,2p,2p),\lambda_1^\vee;SU(3))&=\zeta_2((2p,2p,2p),\lambda_2^\vee;SU(3))\qquad (p\in \mathbb{N}),\\
\zeta_2((2p,2p,2p),\lambda_1^\vee;PU(3))&=\zeta_2((2p,2p,2p),\lambda_2^\vee;PU(3))\\
&=\zeta_2((2p,2p,2p),{\bf 0};PU(3))\qquad (p\in \mathbb{N}).
\end{align*}
In the next section, we will prove certain functional relations for 
$\zeta_2 ({\bf s},\bs{0};PU(3))$ including \eqref{zeta-A2-val}-\eqref{zeta-A2-8}.
\end{example}

\begin{remark}\label{Rem-SS}
In \cite[Section\,5]{SS}, Subbarao and Sitaramachandrarao proposed a problem of evaluating the double series 
$$\sum_{m,n=1}^\infty \frac{(-1)^{m-1}}{m^{k}n^{k}(m+n)^{k}},\ \sum_{m,n=1}^\infty \frac{(-1)^{m+n}}{m^{k}n^{k}(m+n)^{k}}\quad (k\in \mathbb{N}).$$
As for the case of odd $k$, the third-named author evaluated each series in terms of values of $\zeta(s)$ (see \cite{TPAMS,TMCOMP}). The case of even $k$ is still open. It follows from \eqref{A2-varrho} that the above formulas \eqref{su3-2}-\eqref{su3-8} imply certain answers to a problem analogous to that of Subbarao and Sitaramachandrarao.
\end{remark}

\begin{example} \label{exam-A_3}
We consider the $A_3$ type. 
Let $\Delta=\Delta(A_3)$ with $\Psi=\{ \alpha_1,\,\alpha_2,\,\alpha_3\}$, $\Delta_{+}=\{ \alpha_1,\,\alpha_2,\,\alpha_3,\,\alpha_1+\alpha_2,\,\alpha_2+\alpha_3,\,\alpha_1+\alpha_2+\alpha_3\}$, $P=\sum_{j=1}^{3}\mathbb{Z}\lambda_j$ and $Q=\sum_{j=1}^{3}\mathbb{Z}\alpha_j$. Analogously to \eqref{lambda-alpha}, we have
\begin{align}\label{lambda-alpha-A3}
\lambda_1=\frac{3}{4}\alpha_1+\frac{1}{2}\alpha_2+\frac{1}{4}\alpha_3,\quad
\lambda_2=\frac{1}{2}\alpha_1+\alpha_2+\frac{1}{2}\alpha_3,\quad  
\lambda_3=\frac{1}{4}\alpha_1+\frac{1}{2}\alpha_2+\frac{3}{4}\alpha_3.
\end{align}
It is known that $P/Q\simeq \mathbb{Z}/4\mathbb{Z}$ (see \cite{Bourbaki}). 
Therefore there is a unique intermediate lattice $L_1$ with 
$P\supsetneq L_1 \supsetneq Q$, satisfying $(L_1:Q)=2$. 
The group corresponding to $P$ (resp. $Q$) is $SU(4)$ (resp. $PU(4)$).   The group
$G=G(L_1)$ is $SU(4)/\{\pm 1\}$, which is known to be isomorphic to $SO(6)$.

We know (see \cite{KM2}) that 
\begin{align}\label{zeta-A3-P}
 & \zeta_3 ({\bf s},\bs{y};SU(4))=\zeta_3 ({\bf s},\bs{y};P;A_3) \\
   & =\sum_{m_1,m_2,m_3=1}^\infty 
     \frac{e^{2\pi i\langle\bs{y},m_1\lambda_1+m_2\lambda_2+m_3\lambda_3\rangle}}{m_1^{s_1}m_2^{s_2}m_3^{s_3}(m_1+m_2)^{s_4}
     (m_2+m_3)^{s_5}(m_1+m_2+m_3)^{s_6}}. \notag 
\end{align}
Note that $\zeta_3 ({\bf s},\bs{0};SU(4))=\zeta_3 ({\bf s};A_3)$ (see \cite{MT}). 
For example, similarly to \eqref{Ber-A2-gen}, we can compute 
$\mathcal{P}((2,2,2,2,2,2),\bs{y};A_3)$ 
from the generating function which was already given in \cite[Example\,2]{KM5}, though it is too long to write it here. Hence we can obtain from \eqref{gene-Ber} that
$$\mathcal{P}((2,2,2,2,2,2),\lambda_1^\vee;P;A_3)=-\frac{19329337}{14283291230208000},$$
where $\lambda_1^\vee=\frac{3}{4}\alpha_1^\vee+\frac{1}{2}\alpha_2^\vee+\frac{1}{4}\alpha_3^\vee$. 
Therefore we obtain 
from Theorem \ref{thm:W-Z} that
\begin{align}
& \zeta_3((2,2,2,2,2,2),\lambda_1^\vee;SU(4))=\zeta_3((2,2,2,2,2,2),\lambda_1^\vee;P;A_3) \label{A3-lam1} \\
& \qquad =\sum_{m_1,m_2,m_3=1}^\infty 
     \frac{i^{3l+2m+n}}{m_1^{2}m_2^{2}m_3^{2}(m_1+m_2)^{2}
     (m_2+m_3)^{2}(m_1+m_2+m_3)^{2}}\notag \\
& \qquad = -\frac{19329337}{2678117105664000}\pi^{12}. \notag
\end{align}
Concerning $L_1$ and $Q$, similarly to \eqref{Q2}, we can show
\begin{align}
& (L_1)_++\rho =\left\{ \sum_{j=1}^{3}m_j\lambda_j\,\bigg|\,(m_j) \in \mathbb{N}^3,\ m_1 \equiv m_3\ (\text{mod}\ 2)\right\},\label{A3-L}\\
& Q_{+}+\rho  =\left\{ \sum_{j=1}^{3}m_j\lambda_j\,\bigg|\,(m_j) \in \mathbb{N}^3,\ m_1+2m_2+3m_3 \equiv 2\ (\text{mod}\ 4)\right\}.
\label{A3-Q}
\end{align}
In fact, letting $\lambda=\sum_{j=1}^3 m_j\lambda_j\in Q_++\rho$ ($m_j\in\mathbb{N}$), we have $\lambda-\rho=\sum_{j=1}^3 n_j\lambda_j\in Q_+$, where 
$n_j=m_j-1$ ($1\leq j\leq 3$).
From \eqref{lambda-alpha-A3} we have
$$
  \sum_{j=1}^3 n_j\lambda_j=\frac{3n_1+2n_2+n_3}{4}\alpha_1+
   \frac{n_1+2n_2+n_3}{2}\alpha_2+\frac{n_1+2n_2+3n_3}{4}\alpha_3,
$$
which belongs to $Q$.  Therefore
\begin{align*}
&{\rm (i)}\;3n_1+2n_2+n_3\equiv 0\;{\rm (mod\; 4)};\\
&{\rm (ii)}\;n_1+2n_2+n_3\equiv 0\;{\rm (mod\; 2)};\\
&{\rm (iii)}\;n_1+2n_2+3n_3\equiv 0\;{\rm (mod\; 4)}.
\end{align*}
We see that (iii) implies
\begin{align*}
{\rm (iv)}\; n_1\equiv n_3 \;{\rm (mod\; 2)},
\end{align*}
which automatically implies (ii).    Moreover we find that (iii) and (iv)
imply (i).   Therefore the only essential condition is
(iii), which is equivalent to the congruence condition in \eqref{A3-Q}.

Next, define the homomorphism
$\eta:P\cong \mathbb{Z}^3 \to\mathbb{Z}/4\mathbb{Z}$ by
$$
\eta(n_1,n_2,n_3)=n_1+2n_2+3n_3\;{\rm (mod\; 4)}.
$$
Then from the above argument we find that $Q={\rm Ker}\;\eta$.   Let $L_1^*$ be the set of all $(n_1,n_2,n_3)$ 
satisfying $n_1\equiv n_3 \;{\rm (mod\; 2)}$.  Then
$\{0\}\subsetneq \eta(L_1^*)\subsetneq \mathbb{Z}/4\mathbb{Z}$, hence
$Q\subsetneq L_1^* \subsetneq P$.   Therefore $L_1^*$ should be equal to $L_1$,
which implies \eqref{A3-L}.

Let $\bs{y}=y_1\alpha_1^\vee+y_2\alpha_2^\vee+y_3\alpha_3^\vee$. From \eqref{A3-L} and \eqref{A3-Q} we obtain
\begin{align}\label{zeta-A3-L}
  &\zeta_3({\bf s},\bs{y};SO(6))=\zeta_3 ({\bf s},\bs{y};L_1;A_3)\\ 
  &=\sum_{m_1,m_2,m_3=1 \atop m_1\equiv m_3\,(\text{mod}\,2)}^\infty 
     \frac{e^{2\pi i\langle\bs{y},m_1\lambda_1+m_2\lambda_2+m_3\lambda_3\rangle}}{m_1^{s_1}m_2^{s_2}m_3^{s_3}(m_1+m_2)^{s_4} 
     (m_2+m_3)^{s_5}(m_1+m_2+m_3)^{s_6}}, \notag
\end{align}
\begin{align}\label{zeta-A3-Q}
  &\zeta_3({\bf s},\bs{y};PU(4))=\zeta_3 ({\bf s},\bs{y};Q;A_3)\\
 &=\sum_{m_1,m_2,m_3=1 \atop m_1+2m_2+3m_3 \equiv 2\ (\text{mod}\ 4)}^\infty 
     \frac{e^{2\pi i\langle\bs{y},m_1\lambda_1+m_2\lambda_2+m_3\lambda_3\rangle}}{m_1^{s_1}m_2^{s_2}m_3^{s_3}(m_1+m_2)^{s_4}
     (m_2+m_3)^{s_5}(m_1+m_2+m_3)^{s_6}}.\notag
\end{align}
\end{example}

\begin{example} \label{Ex-BC}
We consider the case of $B_r$ and of $C_r$ types. 
The simply-connected group $\widetilde{G}$ in the $B_r$ case is the spinor group
$Spin(2r+1)$, and $\widetilde{G}/\widetilde{Z}=SO(2r+1)$, where $\widetilde{Z}$ is the center of $\widetilde{G}$. In the $C_r$ case
$\widetilde{G}=Sp(r)$, and $\widetilde{G}/\widetilde{Z}$ is the projective symplectic group
$PSp(r)$.
From \cite{KM2,KMTJC,KM3}, we already gave that 
\begin{align*}
  &\zeta_2({\bf s},\bs{y};Spin(5))=\zeta_2 ({\bf s},\bs{y};P;B_2) \\
    &\quad=\sum_{m_1,m_2=1}^\infty 
     \frac{e^{2\pi i\langle\bs{y},m_1\lambda_1+m_2\lambda_2\rangle}}{m_1^{s_1}m_2^{s_2}(m_1+m_2)^{s_3}(2m_1+m_2)^{s_4}},\\
  &\zeta_2({\bf s},\bs{y};Sp(2))=\zeta_2 ({\bf s},\bs{y};P;C_2)\\ 
    &\quad=\sum_{m_1,m_2=1}^\infty 
     \frac{e^{2\pi i\langle\bs{y},m\lambda_1+n\lambda_2\rangle}}{m_1^{s_1}m_2^{s_2}(m_1+m_2)^{s_3}(m_1+2m_2)^{s_4}},\\
  &\zeta_3({\bf s},\bs{y};Spin(7))= \zeta_3({\bf s},\bs{y};P;B_3)\\ 
    &\quad=\sum_{m_1,m_2,m_3=1}^\infty 
     \frac{e^{2\pi i\langle\bs{y},m_1\lambda_1+m_2\lambda_2+m_3\lambda_3\rangle}}{m_1^{s_1}m_2^{s_2}m_3^{s_3}(m_1+m_2)^{s_4}(m_2+m_3)^{s_5}(2m_2+m_3)^{s_6}} \\
& \quad \times \frac{1}{(m_1+m_2+m_3)^{s_7}(m_1+2m_2+m_3)^{s_8}(2m_1+2m_2+m_3)^{s_9}}, \\
  &\zeta_3 ({\bf s},\bs{y};Sp(3))= \zeta_3 ({\bf s},\bs{y};P;C_3)\\ 
    &\quad=\sum_{m_1,m_2,m_3=1}^\infty 
     \frac{e^{2\pi i\langle\bs{y},m_1\lambda_1+m_2\lambda_2+m_3\lambda_3\rangle}}{m_1^{s_1}m_2^{s_2}m_3^{s_3}(m_1+m_2)^{s_4}(m_2+m_3)^{s_5}(m_2+2m_3)^{s_6}} \\
&  \quad \times \frac{1}{(m_1+m_2+m_3)^{s_7}(m_1+m_2+2m_3)^{s_8}(m_1+2m_2+2m_3)^{s_9}}.
\end{align*}
We know that $(P:Q)=2$ in the case of $B_r$ and of $C_r$ types (see \cite{Bourbaki}). 
Therefore the lattice $L$ with $P\supset L \supset Q$ coincides with $P$ or $Q$. 
We show
\begin{align*}
& Q_{+}(B_2)+\rho  =\left\{ \sum_{j=1}^{2}m_j\lambda_j\,\bigg|\,(m_j) \in \mathbb{N}^2,\ m_2 \equiv 1\ (\text{mod}\ 2)\right\},\\
& Q_{+}(C_2)+\rho  =\left\{ \sum_{j=1}^{2}m_j\lambda_j\,\bigg|\,(m_j) \in \mathbb{N}^2,\ m_1 \equiv 1\ (\text{mod}\ 2)\right\},\\
& Q_{+}(B_3)+\rho  =\left\{ \sum_{j=1}^{3}m_j\lambda_j\,\bigg|\,(m_j) \in \mathbb{N}^3,\ m_3 \equiv 1\ (\text{mod}\ 2)\right\},\\
& Q_{+}(C_3)+\rho  =\left\{ \sum_{j=1}^{3}m_j\lambda_j\,\bigg|\,(m_j) \in \mathbb{N}^3,\ m_1 \equiv m_3\ (\text{mod}\ 2)\right\}.
\end{align*}
In fact, we consider, for example, the case of $B_3$. Analogously to Examples \ref{Exam-A_2} and \ref{exam-A_3}, we have
\begin{align}\label{lambda-alpha-B3}
\lambda_1=\alpha_1+\alpha_2+\alpha_3,\quad
\lambda_2=\alpha_1+2\alpha_2+{2}\alpha_3,\quad  
\lambda_3=\frac{1}{2}\alpha_1+\alpha_2+\frac{3}{2}\alpha_3.
\end{align}
Let $\lambda=\sum_{j=1}^3 m_j\lambda_j\in Q_{+}+\rho$ ($m_j\in\mathbb{N}$). Then $\lambda-\rho=\sum_{j=1}^3 n_j\lambda_j\in Q_+$, where $n_j=m_j-1$ ($1\leq j\leq 3$).
it follows from \eqref{lambda-alpha-B3} that
$$
  \lambda-\rho=\frac{2n_1+2n_2+n_3}{2}\alpha_1+
   (n_1+2n_2+n_3)\alpha_2+\frac{2n_1+4n_2+3n_3}{2}\alpha_3,
$$
which belongs to $Q$.  Therefore $n_3\equiv 0 \;{\rm (mod\; 2)}$, that is, 
$m_3\equiv 1 \;{\rm (mod\; 2)}$. 
The cases of $B_2$, $C_2$ and $C_3$ can be similarly treated. 

Therefore we obtain 
\begin{align}
  &\zeta_2({\bf s},\bs{y};SO(5))=\zeta_2 ({\bf s},\bs{y};Q;B_2) \\
    &\quad=\sum_{m_1,m_2=1 \atop m_2 \equiv 1\ (\text{mod}\ 2) }^\infty 
     \frac{e^{2\pi i\langle\bs{y},m_1\lambda_1+m_2\lambda_2\rangle}}{m_1^{s_1}m_2^{s_2}(m_1+m_2)^{s_3}(2m_1+m_2)^{s_4}},\notag\\
  &\zeta_2({\bf s},\bs{y};PSp(2))=\zeta_2 ({\bf s},\bs{y};Q;C_2) \label{PSp-2}\\ 
    &\quad=\sum_{m_1,m_2=1\atop m_1 \equiv 1\ (\text{mod}\ 2)}^\infty 
     \frac{e^{2\pi i\langle\bs{y},m_1\lambda_1+m_2\lambda_2\rangle}}{m_1^{s_1}m_2^{s_2}(m_1+m_2)^{s_3}(m_1+2m_2)^{s_4}},\notag\\
  &\zeta_3({\bf s},\bs{y};SO(7))= \zeta_3({\bf s},\bs{y};Q;B_3)\\ 
    &\quad=\sum_{m_1,m_2,m_3=1 \atop m_3 \equiv 1\ (\text{mod}\ 2)}^\infty 
     \frac{e^{2\pi i\langle\bs{y},m_1\lambda_1+m_2\lambda_2+m_3\lambda_3\rangle}}{m_1^{s_1}m_2^{s_2}m_3^{s_3}(m_1+m_2)^{s_4}(m_2+m_3)^{s_5}(2m_2+m_3)^{s_6}} \notag\\
& \quad \times \frac{1}{(m_1+m_2+m_3)^{s_7}(m_1+2m_2+m_3)^{s_8}(2m_1+2m_2+m_3)^{s_9}}, \notag\\
  &\zeta_3 ({\bf s},\bs{y};PSp(3))= \zeta_3 ({\bf s},\bs{y};Q;C_3)\\ 
    &\quad=\sum_{m_1,m_2,m_3=1 \atop m_1 \equiv m_3\ (\text{mod}\ 2)}^\infty 
     \frac{e^{2\pi i\langle\bs{y},m_1\lambda_1+m_2\lambda_2+m_3\lambda_3\rangle}}{m_1^{s_1}m_2^{s_2}m_3^{s_3}(m_1+m_2)^{s_4}(m_2+m_3)^{s_5}(m_2+2m_3)^{s_6}} \notag\\
&  \quad \times \frac{1}{(m_1+m_2+m_3)^{s_7}(m_1+m_2+2m_3)^{s_8}(m_1+2m_2+2m_3)^{s_9}}.\notag
\end{align}
Now we evaluate special values.   Consider the case $G=PSp(2)$.   
Theorem \ref{thm:W-Z} of $C_2$ type with $\nu={\bf y}={\bf 0}$ gives that 
\begin{align}
  & \zeta_2 ((2k,2l,2l,2k);PSp(2)) \in \mathbb{Q}\cdot \pi^{4(k+l)} \label{W-Z-C2}
\end{align}
for $k,l \in \mathbb{N}$. Actually we have already given the generating function of $C_2$ type (see \cite[Examples 1]{KM5}) and also the generalized Bernoulli function $\mathcal{P}((2,2,2,2),\bs{y};C_2)$ (see \cite[Examples 3]{KM5}). Similarly we can give explicit forms of $\mathcal{P}((2k,2l,2l,2k),\bs{y};C_2)$ $(k,l\in \mathbb{N})$, though it is too complicated to describe them here.   (We will give the explicit form of $\mathcal{P}((2,4,4,2,\bs{y};C_2)$ in Appendix.)   Using these results, we can obtain the following explicit formulas: 
\begin{align*}
&\zeta_2((2,2,2,2);PSp(2))=\frac{\pi^8}{322560},\\
&\zeta_2((2,4,4,2);PSp(2))=\frac{29}{3832012800}\pi^{12},\\
&\zeta_2((4,2,2,4);PSp(2))=\frac{13}{3832012800}\pi^{12},\\
&\zeta_2((4,4,4,4);PSp(2))=\frac{479}{55794106368000}\pi^{16}.
\end{align*}
We will give another type of evaluation formulas in the following sections (see Examples \ref{Ex-C_2} and \ref{Exam-LL}).
\end{example}

\begin{remark}
From \eqref{1-2} and \eqref{wittenequalroot} we see that the original volume
formula of Witten is restricted to the case 
${\bf s}=2(k_{\alpha})_{\alpha\in\Delta_+}$, where all the
$k_{\alpha}$s are the same.   Our Theorem \ref{thm:W-Z} is able to treat a wider
class of special values, such as the $(2,4,4,2)$ and $(4,2,2,4)$ cases in the above.
\end{remark}

\ 

%%%%%%%%%%%%%%%%%%%%%%%%%%%%%%%%%%%%%%%%%%%
\section{Functional relations and various evaluation formulas}\label{sec-5}
%%%%%%%%%%%%%%%%%%%%%%%%%%%%%%%%%%%%%%%%%%%

In the preceding section, we gave explicit forms of several zeta-functions of Lie 
groups, and especially gave some evaluation formulas in the cases of 
$A_2$, $C_2$ and $A_3$ 
types at even integer points, by computing generating functions of their values. 
However, it seems a difficult problem to evaluate zeta-functions of Lie 
groups at arbitrary positive integer points by that method. 
In this section, we give various evaluation formulas for zeta values in the cases of $A_2$ and $C_2(\simeq B_2)$ types, by proving certain functional relations among them which are analogues of our previous results given in \cite{KMT,KMTJC,KM5,MT,TsC}. 
The advantage of the method in this section is that it may treat the special values at
${\bf s}={\bf l}=(l_{\alpha})_{\alpha\in\Delta_+}$, $l_{\alpha}\in{\Bbb N}$ and some
of them are odd.

First we consider 
$\zeta_2 ({\bf s};PU(3))=\zeta_2 ({\bf s},\bs{0};Q;A_2)$ and prove the following theorem, where $\phi(s,\alpha)$ is the Lerch zeta-function defined by \eqref{Lerch-def}. 
. 

\begin{theorem} \label{T-4-1} 
For $p,q \in \mathbb{N}$, 
\begin{equation}
\begin{split}
  & 3\bigg\{ \zeta_2 ((p,q,s);PU(3)) + (-1)^{p} \zeta_2 ((p,s,q);PU(3)) +(-1)^q \zeta_2 ((q,s,p);PU(3)) \bigg\}\\
& =-\sum_{\vv=0}^{p}\binom{p+q-\vv-1}{q-1}(-1)^{\vv}\frac{(2\pi i)^{\vv}}{\vv!}\sum_{a=0}^{2}B_{\vv}\left(\frac{a}{3}\right)\phi(s+p+q-\vv,-a/3) \\
& \ \ -\sum_{\vv=0}^{q}\binom{p+q-\vv-1}{p-1}\frac{(2\pi i)^{\vv}}{\vv!}\sum_{a=0}^{2}B_{\vv}\left(\frac{a}{3}\right)\phi(s+p+q-\vv,a/3) 
\end{split}
\label{T-4-1-eq}
\end{equation}
holds for $s\in \mathbb{C}$ except for singularities of functions on the both sides.
\end{theorem}

\begin{example} \label{Exam-4-4}
It should be emphasized that Theorem \ref{T-4-1} gives evaluation formulas for $\zeta_2 ((a,b,c);PU(3))$ when $a+b+c$ is odd. For example, putting $(p,q,s)=(1,1,1)$ in \eqref{T-4-1-eq}, we have
\begin{equation}
\begin{split}
  3\zeta_2 ((1,1,1);PU(3)) 
& =\sum_{\vv=0}^{1}(-1)^{\vv}\frac{(2\pi i)^{\vv}}{\vv!}\sum_{a=0}^{2}B_{\vv}\left(\frac{a}{3}\right)\sum_{m=1}^\infty \frac{\varrho^{-ma}}{m^{3-\vv}} \\
& \qquad +\sum_{\vv=0}^{1}\frac{(2\pi i)^{\vv}}{\vv!}\sum_{a=0}^{2}B_{\vv}\left(\frac{a}{3}\right)\sum_{m=1}^\infty \frac{\varrho^{ma}}{m^{3-\vv}}. 
\end{split}
\label{A2-L-Value-2}
\end{equation}
We can easily check that
\begin{equation}\label{B_1}
\sum_{a=0}^{2}\varrho^{la}B_1\left(\frac{a}{3}\right)=
\begin{cases}
-\frac{1}{2} & ( l \equiv 0 \ (\text{mod\ }3)) \\
-\frac{1}{2}-\frac{1}{2\sqrt{3}}i  & ( l \equiv 1 \ (\text{mod\ }3)) \\
-\frac{1}{2}+\frac{1}{2\sqrt{3}}i  & ( l \equiv 2 \ (\text{mod\ }3)). 
\end{cases}
\end{equation}
Then \eqref{A2-L-Value-2} can be rewritten to 
\begin{align}\label{1,1,1}
\zeta_2 ((1,1,1);PU(3))& =\frac{2}{27}\zeta(3)+\frac{2\pi}{3\sqrt{3}}L(2,\chi_3), 
\end{align}
where we denote by $\chi_3$ the primitive Dirichlet character of conductor $3$.
This is an analogue of $\zeta_2 ((1,1,1);SU(3))=2\zeta(3)$ (see \cite{To}). 
Similarly, setting $(p,q,s)=(1,2,2)$ in \eqref{T-4-1-eq}, and using the relations 
\eqref{B_1} and
\begin{equation}\label{B_2}
\sum_{a=0}^{2}\varrho^{la}B_2\left(\frac{a}{3}\right)=
\begin{cases}
\frac{1}{18} & ( l \equiv 0 \ (\text{mod\ }3)) \\
\frac{2}{9} & ( l \equiv 1,2 \ (\text{mod\ }3)),
\end{cases}
\end{equation}
we can obtain
\begin{align}\label{2,2,2}
\zeta_2 ((2,2,1);PU(3))& =-\frac{1}{81}\zeta(5)+\frac{35\pi^2}{243}\zeta(3)-\frac{2\pi}{3\sqrt{3}}L(4,\chi_3).
\end{align}
The above formulas \eqref{1,1,1} and \eqref{2,2,2} can also be deduced by using
\begin{align}\label{phiequalL}
\phi(s,1/3)-\phi(s,2/3)=2i\sum_{m=1}^\infty \frac{\sin(2\pi m/3)}{m^s}=
\sqrt{3}iL(s,\chi_3)
\end{align}
instead of \eqref{B_1}, \eqref{B_2}.
A more general result will be given in the next section (see Theorem \ref{Th-6-2}). 

By the partial fraction decomposition, we have
\begin{align*}
\zeta_2 ((1,1,1);PU(3))&=\sum_{m,n=1 \atop m\equiv n \ (\text{mod\ }3)}^\infty \frac{1}{mn(m+n)}=2\sum_{m,n=1 \atop m\equiv n \ (\text{mod\ }3)}^\infty \frac{1}{m(m+n)^2}.
\end{align*}
Hence combining with \eqref{1,1,1} we obtain 
\begin{equation}
\sum_{m,n=1 \atop m\equiv n \ (\text{mod\ }3)}^\infty \frac{1}{m(m+n)^2}=\frac{1}{27}\zeta(3)+\frac{\pi}{3\sqrt{3}}L(2,\chi_3). \label{A2-Value}
\end{equation}
This can be regarded as a formula for a partial sum of the double zeta value, analogously to the well-known result given by Euler (cf. \cite{Kaneko}):
$$\sum_{m,n=1}^\infty \frac{1}{m(m+n)^2}=\zeta(3).$$
\end{example}

\begin{remark}
Setting $(p,q,s)=(2k,2k,2k)$ in Theorem \ref{T-4-1} and using the fact
\begin{equation}                                                                       
B_j (x) = -\frac{j!}{(2\pi i)^j}                                                        
\lim_{M \rightarrow \infty} \sum_{\substack{ m=-M \\ m \ne 0}}^{M} \frac{e^{2\pi imx}}
{m^j}  \qquad (j \in {\mathbb{N}};\ 0\leq x<1)                                        
\label{eq:berfu}                                                                       
\end{equation}
(see \cite[p.\,266]{Ap}), we obtain
\begin{align}                                                                          
  & \zeta_2 ((2k,2k,2k);PU(3)) \label{Wit-Pu}\\                                        
  &=\frac{(2\pi i)^{6k}}{9}\sum_{\vv=0}^{2k}\binom{4k-\vv-1}{2k-1}\sum_{a=0}^{2}
\frac{B_{\vv}\left(\frac{a}{3}\right)}{\vv!}\frac{B_{6k-\vv}\left(\frac{a}{3}\right)}
{(6k-\vv)!}\quad (k\in \mathbb{N}), \notag                                             
\end{align}
which is an explicit form of \eqref{thm:W-Z} for $PU(3)$ and includes
\eqref{zeta-A2-val}-\eqref{zeta-A2-8}.
\end{remark}

Now we give the proof of Theorem \ref{T-4-1}. 
We first prepare the following lemma which can be proved by the same method as introduced in \cite{KMTJC}. In fact, this lemma in the case when $p$ and $q$ are even has already been proved in \cite[(7.55)]{KMTJC}. We use the notation $\phi(s):=\phi(s,1/2)=(2^{1-s}-1)\zeta(s)$ and $\la_m:=(1+(-1)^m)/2$ for $m\in \mathbb{Z}$. 

\begin{lemma}\label{Lem-5-1} 
For $p\in \mathbb{N}$, $s \in \mathbb{R}$ with $s>1$ and $x\in \mathbb{C}$ with $|x|\leq 1$,
\begin{equation}
\begin{split}
& \sum_{l\not=0,\,m\geq 1 \atop l+m\not=0} \frac{(-1)^{l+m}x^m e^{i(l+m)\theta}}{l^{p}m^{s}(l+m)^{q}} \\
& \ \ -2\sum_{j=0}^{p}\ \phi(p-j)\varepsilon_{p-j} \sum_{\xi=0}^{j}\binom{q-1+j-\xi}{q-1}(-1)^{j-\xi}\sum_{m=1}^\infty \frac{(-1)^{m}x^m e^{im\theta}}{m^{s+q+j-\xi}}\frac{(i\theta)^{\xi}}{\xi!} \\
& \ \ +2\sum_{j=0}^{q}\ \phi(q-j)\varepsilon_{q-j} \sum_{\xi=0}^{j}\binom{p-1+j-\xi}{p-1}(-1)^{p-1}\sum_{m=1}^\infty \frac{x^m }{m^{s+p+j-\xi}}\frac{(i\theta)^{\xi}}{\xi!}=0 
\end{split}
\label{eq-5-6}
\end{equation}
holds for $\theta \in [-\pi,\pi]$. 
\end{lemma}

\begin{proof}
For $p \in \mathbb{N}$, 
it is known that (see, for example, \cite[(4.31),\,(4.32)]{KMTJC})
\begin{align}
& \lim_{L \to \infty}\sum_{-L\leq l\leq L \atop l\not=0} \frac{(-1)^{l}e^{il\theta}}{l^{p}}=2\sum_{j=0}^{p}\ \phi(p-j)\varepsilon_{p-j}\frac{(i\theta)^{j}}{j!}\quad (\theta \in (-\pi,\pi)). \label{eq-4-4-2} 
\end{align}
Note that the left-hand side is uniformly convergent for $\theta \in (-\pi,\pi)$ (see \cite[$\S$ 3.35]{WW}), and is also absolutely convergent for $\theta \in [-\pi,\pi]$ when $p\geq 2$. 
First we assume $p\geq 2$. Then, for $\theta \in [-\pi,\pi]$, it follows from \eqref{eq-4-4-2} that
\begin{align}
& \left( \sum_{l\in \mathbb{Z} \atop l\not=0} \frac{(-1)^{l}e^{il\theta}}{l^{p}}-2\sum_{j=0}^{p}\ \phi(p-j)\varepsilon_{p-j}\frac{(i\theta)^{j}}{j}\right)\sum_{m=1}^\infty \frac{(-1)^{m}x^m e^{im\theta}}{m^s}=0, \label{eq-4-4-3} 
\end{align}
where the left-hand side is absolutely and uniformly convergent for $\theta \in [-\pi,\pi]$. Therefore we have
\begin{equation}
\begin{split}
& \sum_{l\in \mathbb{Z},\ l\not=0 \atop{m\geq 1 \atop l+m\not=0}} \frac{(-1)^{l+m}x^m e^{i(l+m)\theta}}{l^{p}m^{s}}-2\sum_{j=0}^{p}\ \phi(p-j)\varepsilon_{p-j}\left\{ \sum_{m=1}^\infty \frac{(-1)^{m}x^m e^{im\theta}}{m^s}\right\} \frac{(i\theta)^{j}}{j!}\\
& \ \ \ \ =(-1)^{p+1}\sum_{m=1}^\infty \frac{x^m}{m^{s+p}} 
\end{split}
\label{eq-4-4-4} 
\end{equation}
for $\theta \in [-\pi,\pi]$. Now we apply \cite[Lemma 6.2]{KMTJC} with $d=q\in \mathbb{N}$. Then we obtain \eqref{eq-5-6} for $p\geq 2$. 

Next we prove the case $p=1$. As we proved above, \eqref{eq-5-6} in the case $p=2$ holds. Replacing $x$ by $-xe^{i\theta}$ in this case, we have
\begin{equation}
\begin{split}
& \sum_{l\not=0,\,m\geq 1 \atop l+m\not=0} \frac{(-1)^{l}x^m e^{il\theta}}{l^{2}m^{s}(l+m)^{q}} \\
& \ \ -2\sum_{j=0}^{2}\ \phi(2-j)\varepsilon_{2-j} \sum_{\xi=0}^{j}\binom{q-1+j-\xi}{q-1}(-1)^{j-\xi}\sum_{m=1}^\infty \frac{x^m}{m^{s+q+j-\xi}}\frac{(i\theta)^{\xi}}{\xi!} \\
& \ \ +2\sum_{j=0}^{q}\ \phi(q-j)\varepsilon_{q-j} \sum_{\xi=0}^{j}\binom{1+j-\xi}{1}(-1)^{1}\sum_{m=1}^\infty \frac{(-1)^mx^m e^{-im\theta}}{m^{s+2+j-\xi}}\frac{(i\theta)^{\xi}}{\xi!}=0 
\end{split}
\label{eq-5-62}
\end{equation}
for $\theta \in [-\pi,\pi]$. We denote the first, the second and the third term on the left hand side of \eqref{eq-5-62} by $I_1(\theta)$, $I_2(\theta)$ and $I_3(\theta)$, respectively. 
We differentiate these terms in $\theta$. We can easily compute $I_1'(\theta)$ and $I_2'(\theta)$. As for $I_3'(\theta)$, we have
\begin{align*}
I_3'(\theta)& = 2\sum_{j=0}^q\phi(q-j)\varepsilon_{q-j}\bigg\{-i\sum_{\xi=0}^{j}(1+j-\xi)(-1)\sum_{m=1}^\infty \frac{(-1)^mx^m e^{-im\theta}}{m^{s+1+j-\xi}}\frac{(i\theta)^{\xi}}{\xi!}\\
& \qquad +i\sum_{\xi=1}^{j}(1+j-\xi)(-1)\sum_{m=1}^\infty \frac{(-1)^mx^m e^{-im\theta}}{m^{s+2+j-\xi}}\frac{(i\theta)^{\xi-1}}{(\xi-1)!}\bigg\}.
\end{align*}
Note that as for the second member in the curly brackets on the right-hand side, $\xi$ may also run from $1$ to $j+1$ because $1+j-(j+1)=0$ in the summand. Hence, by replacing $\xi-1$ by $\xi$, we have 
\begin{align*}
I_3'(\theta)& = 2i\sum_{j=0}^q\phi(q-j)\varepsilon_{q-j}\sum_{\xi=0}^{j}\sum_{m=1}^\infty \frac{(-1)^mx^m e^{-im\theta}}{m^{s+1+j-\xi}}\frac{(i\theta)^{\xi}}{\xi!}.
\end{align*}
Thus, we see that $(I_1'(\theta)+I_2'(\theta)+I_3'(\theta))/i$,
replacing $x$ by $-xe^{i\theta}$, 
gives \eqref{eq-5-6} in the case $p=1$. This completes the proof.
\end{proof}

Here we quote the following lemma given in \cite[Lemma 9.1]{KM5}. Note that 
the assertion in \cite[Lemma 9.1]{KM5} is stated only in the case that $p$ is even. However, we can easily check that the assertion holds for any $p\in \mathbb{N}$ as follows. 

\begin{lemma} \label{L-4-3} 
Let $t\in[0,2\pi) \subset \mathbb{R}$, and
$h:\mathbb{N}_{0} \to \mathbb{C}$ be a function (which may depend on $t$). 
Then, for $p \in \mathbb{N}$,
\begin{equation}
\begin{split}
& \sum_{j=0}^{p} \phi(p-j) \varepsilon_{p-j}\sum_{\xi=0}^{j} h(j-\xi)\frac{(i(t-\pi))^{\xi}}{\xi!} =-\frac{1}{2}\sum_{\xi=0}^{p}h(p-\xi)\frac{(2\pi i)^{\xi}}{\xi!}B_{\xi}\left(\left\{ \frac{t}{2\pi}\right\} \right). 
\end{split}
\label{eq-5-6-2}
\end{equation}
\end{lemma}

Put $\theta=t-\pi$ $(0\leq t < 2\pi)$ in \eqref{eq-5-6} and multiply by $(-1)^p$ the both sides. Then, using Lemma \ref{L-4-3}, we have the following. Note that this can be derived by a certain transformation of a result of Nakamura \cite[Theorem 3.1]{Na3} when $|x|=1$.

\begin{lemma} \label{L-4-4} 
For $p,q \in \mathbb{N}$, $s,t\in \mathbb{R}$ with $s>1$ and $t\in [0,2\pi)$, and $x\in \mathbb{C}$ with $|x|\leq 1$, 
\begin{equation}
\begin{split}
& \sum_{l,m=1}^\infty \frac{x^{l+m} e^{imt}}{l^{p}m^{q}(l+m)^{s}}+(-1)^p\sum_{l,m=1}^\infty \frac{x^m e^{i(l+m)t}}{l^{p}m^{s}(l+m)^{q}}+(-1)^{q}\sum_{l,m=1}^\infty \frac{x^me^{-ilt}}{l^{q}m^{s}(l+m)^{p}}\\
& =-\sum_{\vv=0}^{p}\binom{p+q-\vv-1}{q-1}(-1)^{\vv}\sum_{m=1}^\infty \frac{x^me^{imt}}{m^{s+p+q-\vv}} \frac{(2\pi i)^{\vv}}{\vv!}B_{\vv}\left(\left\{ \frac{t}{2\pi}\right\}\right) \\
& \quad -\sum_{\vv=0}^{q}\binom{p+q-\vv-1}{p-1}\sum_{m=1}^\infty \frac{x^m}{m^{s+p+q-\vv}} \frac{(2\pi i)^{\vv}}{\vv!}B_{\vv}\left(\left\{ \frac{t}{2\pi}\right\}\right). 
\end{split}
\label{eq-5-7} 
\end{equation}
\end{lemma}

Using these results, we give the proof of Theorem \ref{T-4-1} as follows.

\begin{proof}[Proof of Theorem  \ref{T-4-1}] 
Let $x=e^{-2it}$ and further let $t=2\pi a/3$ $(a=0,1,2)$ on the both sides of \eqref{eq-5-7}. Then, summing up with $a=0,1,2$ and using the fact for $\varrho=e^{2\pi i/3}$ that 
\begin{equation*}
\sum_{a=0}^{2} \varrho^{Na}=
\begin{cases} 
3 & \text{$N \equiv 0\ (\text{mod\ }3)$} \\
0 & \text{$N \not\equiv 0\ (\text{mod\ }3)$}, 
\end{cases}
\end{equation*}
we have
\begin{align*}
& 3\bigg\{\sum_{l,m\geq 1\atop l\equiv m\,({\rm mod\,3})} \frac{1}{l^{p}m^{q}(l+m)^{s}}+(-1)^p\sum_{l,m\geq 1\atop l\equiv m\,({\rm mod\,3})} \frac{1}{l^{p}m^{s}(l+m)^{q}}\\
& \qquad +(-1)^{q}\sum_{l,m\geq 1\atop l\equiv m\,({\rm mod\,3})} \frac{1}{l^{q}m^{s}(l+m)^{p}}\bigg\}\\
& =-\sum_{\vv=0}^{p}\binom{p+q-\vv-1}{q-1}(-1)^{\vv}\sum_{a=0}^{2}\sum_{m=1}^\infty \frac{\varrho^{-ma}}{m^{s+p+q-\vv}} \frac{(2\pi i)^{\vv}}{\vv!}B_{\vv}\left(\left\{ \frac{a}{3}\right\}\right) \\
& \ \ -\sum_{\vv=0}^{q}\binom{p+q-\vv-1}{p-1}\sum_{a=0}^{2}\sum_{m=1}^\infty \frac{\varrho^{ma}}{m^{s+p+q-\vv}} \frac{(2\pi i)^{\vv}}{\vv!}B_{\vv}\left(\left\{ \frac{a}{3}\right\}\right). 
\end{align*}
Noting  \eqref{Lerch-zeta-A2} and using Proposition \ref{Prop-conti},
we complete the proof of Theorem \ref{T-4-1}. 
\end{proof}

Secondly we consider the case of $C_2$ type, namely the zeta-function
\begin{align*}
  \zeta_2({\bf s};PSp(2))&=\zeta_2 ({\bf s},\bs{0};Q;C_2) \\ 
    &=\sum_{m_1,m_2=1\atop m_1 \equiv 1\ (\text{mod}\ 2)}^\infty 
     \frac{1}{m_1^{s_1}m_2^{s_2}(m_1+m_2)^{s_3}(m_1+2m_2)^{s_4}},
\end{align*}
defined by \eqref{PSp-2} with ${\bf y}={\bf 0}$. 
Here, we aim to prove a $C_2$ type analogue of Lemma \ref{L-4-4}. It is noted that 
we already studied the zeta-function of $C_2$ type in \cite[Section 8]{KMTJC} and \cite[Section 9]{KM5}. In fact, 
using the same method as in the proof of \cite[(9.8)]{KM5}, we can obtain
\begin{align}
& \sum_{l\geq 1 \atop {m\geq 1}} \frac{e^{ilt}}{l^{p}m^s(l+m)^{q}(l+2m)^{r}}+\sum_{l\geq 1 \atop {m\geq 1 \atop {l\not=m \atop l\not=2m}}} \frac{e^{-ilt}}{(-l)^{p}m^s(-l+m)^{q}(-l+2m)^{r}} \label{9-8}\\
& \ +\sum_{\xi=0}^{p} \sum_{\omega=0}^{p-\xi} \binom{\omega+r-1}{\omega}\binom{p+q-1-\xi-\omega}{q-1} \frac{(-1)^{p-\xi}}{2^{r+\omega}}\frac{(2\pi i)^\xi}{\xi!}  \notag\\
& \hspace{0.5in} \times \zeta(s+p+q+r-\xi)B_\xi\left({t}/{2\pi}\right)\notag\\
& \ +\sum_{\xi=0}^{q} \sum_{\omega=0}^{q-\xi} \binom{\omega+r-1}{\omega}\binom{p+q-1-\xi-\omega}{p-1} {(-1)^{p-\omega}}\frac{(2\pi i)^\xi}{\xi!}  \notag\\
& \hspace{0.5in} \times \phi(s+p+q+r-\xi,-t/2\pi)B_\xi\left({t}/{2\pi}\right)\notag\\
& \ +\sum_{\xi=0}^{r} \sum_{\omega=0}^{p-1} \binom{\omega+r-\xi}{\omega}\binom{p+q-2-\omega}{q-1} \frac{(-1)^{p-1}}{2^{r-\xi+\omega+1}}\frac{(2\pi i)^\xi}{\xi!}  \notag\\
& \hspace{0.5in} \times \phi(s+p+q+r-\xi,-t/\pi)B_\xi\left({t}/{2\pi}\right)\notag\\
& \ +\sum_{\xi=0}^{r} \sum_{\omega=0}^{q-1} \binom{\omega+r-\xi}{\omega}\binom{p+q-2-\omega}{p-1} {(-1)^{p-\omega}}\frac{(2\pi i)^\xi}{\xi!}  \notag\\
& \hspace{0.5in} \times \phi(s+p+q+r-\xi,-t/\pi)B_\xi\left({t}/{2\pi}\right)=0 \notag
\end{align}
for $p,q,r \in \mathbb{N}$ and $s,t \in \mathbb{R}$ with $s>1$ and $t \in [0,2\pi)$. Actually, this equation with replacing $(p,q,r)$ by $(2p,2q,2p)$ coincides with \cite[(9.8)]{KM5} in the case $(\eta,\rho,\delta,\tau)=(t,0,0,0)$. 
As for the second term on the left-hand side of \eqref{9-8}, 
we split the sum into two parts 
according to the conditions $l<m$ or $l>m$, and transform variables as
$j=m-l(\geq 1)$ when $l<m$, and $j=l-m(\geq 1)$ when $l>m$.   In the latter case
we further split the sum according to $j<m$ or $j>m$ (that is, $l<2m$ or $l>2m$). 
Then we can see that \eqref{9-8} implies the following. 
\begin{align}
& \sum_{l\geq 1 \atop {m\geq 1}} \frac{e^{ilt}}{l^{p}m^s(l+m)^{q}(l+2m)^{r}}+
(-1)^p\sum_{l\geq 1 \atop {m\geq 1}} \frac{e^{-ilt}}{l^{p}m^q(l+m)^{s}(l+2m)^{r}}\label{C2-eq-1}\\
& +(-1)^{p+q}\sum_{l\geq 1 \atop {m\geq 1}} \frac{e^{-i(l+2m)t}}{l^{r}m^q(l+m)^{s}(l+2m)^{p}}+(-1)^{p+q+r}\sum_{l\geq 1 \atop {m\geq 1}} \frac{e^{-i(l+2m)t}}{l^{r}m^s(l+m)^{q}(l+2m)^{p}}\notag\\
& \ +\sum_{\xi=0}^{p} \sum_{\omega=0}^{p-\xi} \binom{\omega+r-1}{\omega}\binom{p+q-1-\xi-\omega}{q-1} \frac{(-1)^{p-\xi}}{2^{r+\omega}}\frac{(2\pi i)^\xi}{\xi!}  \notag\\
& \hspace{0.5in} \times \zeta(s+p+q+r-\xi)B_\xi\left({t}/{2\pi}\right)\notag\\
& \ +\sum_{\xi=0}^{q} \sum_{\omega=0}^{q-\xi} \binom{\omega+r-1}{\omega}\binom{p+q-1-\xi-\omega}{p-1} {(-1)^{p-\omega}}\frac{(2\pi i)^\xi}{\xi!}  \notag\\
& \hspace{0.5in} \times \phi(s+p+q+r-\xi,-t/2\pi)B_\xi\left({t}/{2\pi}\right)\notag\\
& \ +\sum_{\xi=0}^{r} \sum_{\omega=0}^{p-1} \binom{\omega+r-\xi}{\omega}\binom{p+q-2-\omega}{q-1} \frac{(-1)^{p}}{2^{r-\xi+\omega+1}}\frac{(2\pi i)^\xi}{\xi!}  \notag\\
& \hspace{0.5in} \times \phi(s+p+q+r-\xi,-t/\pi)B_\xi\left({t}/{2\pi}\right)\notag\\
& \ +\sum_{\xi=0}^{r} \sum_{\omega=0}^{q-1} \binom{\omega+r-\xi}{\omega}\binom{p+q-2-\omega}{p-1} {(-1)^{p-\omega+1}}\frac{(2\pi i)^\xi}{\xi!}  \notag\\
& \hspace{0.5in} \times \phi(s+p+q+r-\xi,-t/\pi)B_\xi\left({t}/{2\pi}\right)=0. \notag
\end{align}
Denote the left-hand side by $H(t)$, namely \eqref{C2-eq-1} implies $H(t)=0$ for any $t\in [0,2\pi)$. Let $t=0,\pi$. We note that $e^{\pm il\pi}=e^{-i(l+2m)\pi}=(-1)^l$ for $l,m\in \mathbb{N}$. Also we have $\phi(s,0)=\phi(s,-1)=\zeta(s)$, $\phi(s,- 1/2)=\phi(s)=(2^{1-s}-1)\zeta(s)$. Therefore, considering $(H(0)-H(\pi))/2=0$ and noting \eqref{PSp-2}, we have the following result by the meromorphic continuation similarly to Theorem \ref{T-4-1}. 

\begin{theorem} \label{T-PSp-2}
For $p,q,r \in \mathbb{N}$,
\begin{align}
& \zeta_2((p,s,q,r);PSp(2))+(-1)^p\zeta_2((p,q,s,r);PSp(2)) \label{C2-Fq}\\
& \quad +(-1)^{p+q}\zeta_2((r,q,s,p);PSp(2))+(-1)^{p+q+r}\zeta_2((r,s,q,p);PSp(2))\notag\\
& \ +\sum_{\xi=0}^{p} \sum_{\omega=0}^{p-\xi} \binom{\omega+r-1}{\omega}\binom{p+q-1-\xi-\omega}{q-1} \frac{(-1)^{p-\xi}}{2^{r+\omega}}\frac{(2\pi i)^\xi}{\xi!}  \notag\\
& \hspace{0.5in} \times \zeta(s+p+q+r-\xi)\frac{B_\xi(0)-B_\xi\left(1/2\right)}{2}\notag\\
& \ +\sum_{\xi=0}^{q} \sum_{\omega=0}^{q-\xi} \binom{\omega+r-1}{\omega}\binom{p+q-1-\xi-\omega}{p-1} {(-1)^{p-\omega}}\frac{(2\pi i)^\xi}{\xi!}  \notag\\
& \hspace{0.5in} \times \zeta(s+p+q+r-\xi)\frac{B_\xi(0)-\left(2^{1-s-p-q-r+\xi}-1\right)B_\xi\left(1/2\right)}{2}\notag\\
& \ +\sum_{\xi=0}^{r} \sum_{\omega=0}^{p-1} \binom{\omega+r-\xi}{\omega}\binom{p+q-2-\omega}{q-1} \frac{(-1)^{p}}{2^{r-\xi+\omega+1}}\frac{(2\pi i)^\xi}{\xi!}  \notag\\
& \hspace{0.5in} \times \zeta(s+p+q+r-\xi)\frac{B_\xi(0)-B_\xi\left(1/2\right)}{2}\notag\\
& \ +\sum_{\xi=0}^{r} \sum_{\omega=0}^{q-1} \binom{\omega+r-\xi}{\omega}\binom{p+q-2-\omega}{p-1} {(-1)^{p-\omega+1}}\frac{(2\pi i)^\xi}{\xi!}  \notag\\
& \hspace{0.5in} \times \zeta(s+p+q+r-\xi)\frac{B_\xi(0)-B_\xi\left(1/2\right)}{2}=0 \notag
\end{align}
holds for $s \in \mathbb{C}$ except for singularities.
\end{theorem}

\begin{example} \label{Ex-C_2}
By Theorem \ref{T-PSp-2}, we can evaluate $\zeta_2((a,b,c,d);PSp(2))$ in some case when $a+b+c+d$ is odd. For example, setting $(p,s,q,r)=(2,1,1,1)$ and $(2,3,3,5)$ in \eqref{C2-Fq}, we have
\begin{align*}
& \zeta_2((2,1,1,1);PSp(2))=\frac{3}{8}\zeta(2)\zeta(3)-\frac{31}{64}\zeta(5),\\
& \zeta_2((2,3,3,5);PSp(2))= -\frac{15}{16}\zeta(4)\zeta(9) - \frac{17379}{4096}\zeta(2)\zeta(11) + \frac{8191}{1024}\zeta(13).
\end{align*}
In general, it seems to be difficult to evaluate $\zeta_2((a,b,c,d);PSp(2))$ for arbitrary $a,b,c,d\in \mathbb{N}$. We will further consider this problem in the next section. 
\end{example}

\begin{remark} \label{Rem-C2-1}
Putting $(p,s,q,r)=(2k,2l,2l,2k)$ $(k,l\in \mathbb{N})$ in \eqref{C2-Fq}, we can see 
that
$$\zeta_2 ((2k,2l,2l,2k);PSp(2))$$
can be expressed as a polynomial in $\zeta(4k+4l-\xi)(i\pi)^\xi$ with 
$\mathbb{Q}$-coefficients. Since $\zeta_2 ((2k,2l,2l,2k);PSp(2))\in \mathbb{R}$, we 
see that the part consisting of the terms of $\zeta(4k+4l-\xi)(i\pi)^\xi$ for odd $\xi$ vanish. 
On the other hand, for even $\xi$, each term belongs to $\mathbb{Q}\cdot \pi^{4(k+l)}$. Thus we recover \eqref{W-Z-C2}.
%\begin{align*}                                                                         
%  & \zeta_2 ((2k,2l,2l,2k);PSp(2)) \in \mathbb{Q}\cdot \pi^{4(k+l)}\ \ 
%(k,l \in \mathbb{N}).  
%\end{align*}
\end{remark}

\ 

\section{Parity results}\label{sec-6}

Euler proved that the double zeta value (of weight $p+q$)
$$\zeta_2(p,q)=\sum_{m\geq 1 \atop n\geq 1} \frac{1}{m^{p}(m+n)^q}\qquad (p,q \in \mathbb{N};\,q\geq 2)$$
can be expressed as a polynomial in $\{\zeta(j+1)\,|\,j\in \mathbb{N}\}$ with $\mathbb{Q}$-coefficients (cf. \cite{Kaneko}) if its weight $p+q$ is odd. This property is often called the \textit{parity result}, and has been generalized to that for multiple zeta-values (see \cite{IKZ,TsActa}). 

It is an interesting problem to ask what kind of multiple zeta values has this type
of properties.   Tornheim (see \cite[Theorem 7]{To}) proved that 
$\zeta_2((a,b,c),{\bf 0};SU(3))=\zeta_2((a,b,c),{\bf 0};P,A_2)$ has this property, 
that is, it can be expressed as a polynomial in $\{ \zeta(j+1)\,|\,j\in \mathbb{N}\}$ 
with $\mathbb{Q}$-coefficients if its weight $a+b+c$ is odd. 
This result has been generalized by the third-named author \cite{TsPAMS} to the case 
of multiple Mordell-Tornheim zeta values.
Also the third-named author (see \cite{TsArch}) proved that $\zeta_2((a,b,c,d),{\bf 0};Sp(2))=\zeta_2((b,a,c,d),{\bf {0}};Spin(5))$ has this property, which is an extension of the result of Apostol and Vu \cite{ApVu}. 
%From these results, it is interesting to confirm what kind of double Dirichlet series has this property. 

In this section we first prove the following fact, which is a $PU(3)$ type analogue 
of Tornheim's result stated above. 

\begin{theorem} \label{Th-6-2}
Let $a,b,c \in \mathbb{N}$. If $a+b+c$ is odd then 
$\zeta_2 ((a,b,c);PU(3))$ can be expressed as a polynomial in $\{ \phi(j;a/3)\,|\,a\in\{0,1,2\},\ j\in \mathbb{N}\}$ with $\mathbb{Q}[\pi, i]$-coefficients.
\end{theorem}

\begin{proof}
Denote by $\mathfrak{X}$ the set of polynomials in $\{ \phi(j;a/3)\,|\,a\in\{0,1,2\},\ j\in \mathbb{N}\}$ with $\mathbb{Q}[\pi, i]$-coefficients. 
Then we see that the right-hand side of \eqref{T-4-1-eq} with $(p,q,s)=(c,a,b)$ is in $\mathfrak{X}$. 
First we consider the case $a$ is odd and $b,c$ is even, namely $a+b+c$ is odd.
Then, by \eqref{T-4-1-eq}, we have 
\begin{align*}
& \zeta_2 ((c,a,b);PU(3)) +\zeta_2 ((c,b,a);PU(3)) - \zeta_2((a,b,c);PU(3))\in \mathfrak{X}.
\end{align*}
Also, setting $(p,q,s)=(c,b,a)$ in \eqref{T-4-1-eq}, we have
\begin{align*}
& \zeta_2 ((c,a,b);PU(3)) +\zeta_2 ((c,b,a);PU(3)) + \zeta_2 ((b,a,c);PU(3))\in \mathfrak{X}.
\end{align*}
Note that $\zeta_2 ((p,q,r);PU(3))=\zeta_2 ((q,p,r);PU(3))$. Hence, these imply the assertion $\zeta_2((a,b,c);PU(3))\in \mathfrak{X}$. 
As for other cases, we can similarly prove their assertions. 
\end{proof}

\begin{example} \label{Exam-C2-3}
Setting $(p,q,s)=(1,3,5)$ and $(1,5,3)$ in \eqref{T-4-1-eq}, we have 
\begin{align*}
& \zeta_2 ((1,3,5);PU(3)) -\zeta_2 ((1,5,3);PU(3)) - \zeta_2((3,5,1);PU(3))\\
& \quad=-\frac{4}{3}\zeta(9) + \frac{\pi^2}{9}\zeta(7)-\frac{4}{3}(\phi(9,1/3)+\phi(9,2/3)) +\frac{2\pi i}{9}(\phi(8,1/3)-\phi(8,2/3))\\
& \qquad - \frac{\pi^2}{27}(\phi(7,1/3)+\phi(7,2/3)) + \frac{4\pi^3 i}{243}(\phi(6,1/3)-\phi(6,2/3)),\\
& \zeta_2 ((1,5,3);PU(3)) -\zeta_2 ((1,3,5);PU(3)) - \zeta_2((5,3,1);PU(3))\\
& \quad=-2\zeta(9) + \frac{\pi^2}{9}\zeta(7)+\frac{\pi^4}{135}\zeta(5)-2(\phi(9,1/3)+\phi(9,2/3)) \\
& \qquad + \frac{2\pi i}{9}(\phi(8,1/3)-\phi(8,2/3)) - \frac{\pi^2}{27}(\phi(7,1/3)+\phi(7,2/3)) \\
& \qquad + \frac{4\pi^3 i}{243}(\phi(6,1/3)-\phi(6,2/3))-\frac{13\pi^4}{3645}(\phi(5,1/3)+ \phi(5,2/3))\\
& \qquad +\frac{4\pi^5 i}{2187}(\phi(4,1/3)-\phi(4,2/3)).
\end{align*}
Combining these results and noting $\zeta_2((3,5,1);PU(3))=\zeta_2((5,3,1);PU(3))$, we have 
\begin{align*}
\zeta_2((3,5,1);PU(3))& =\frac{5}{3}\zeta(9) - \frac{\pi^2}{9}\zeta(7)-\frac{\pi^4}{270}\zeta(5)\\
& +\frac{5}{3}(\phi(9,1/3)+\phi(9,2/3)) - \frac{2\pi i}{9}(\phi(8,1/3)-\phi(8,2/3))\\
& + \frac{\pi^2}{27}(\phi(7,1/3)+\phi(7,2/3)) - \frac{4\pi^3 i}{243}(\phi(6,1/3)-\phi(6,2/3))\\
& +\frac{13\pi^4}{7290}(\phi(5,1/3)+\phi(5,2/3))-\frac{2\pi^5 i}{2187}(\phi(4,1/3)-\phi(4,2/3)).
\end{align*}
Moreover, since
$$\phi(s,1/3)-\phi(s,2/3)=\sqrt{3}i L(s,\chi_3),\quad \phi(s,1/3)+\phi(s,2/3)=(3^{1-s}-1)\zeta(s)$$
as noted in Example \ref{Exam-4-4}, we find that $\zeta_2((3,5,1);PU(3))$ can actually
be written in terms of Riemann-zeta values and values of the Dirichlet $L$-function
attached to $\chi_3$.
\end{example}

Next consider the $C_2$ case.   Since we have already known the parity result for
$\zeta_2 ((a,b,c,d);Sp(2))$ (\cite{TsArch}), it should be not surprising to know
that the following parity result for $\zeta_2 ((a,b,c,d);PSp(2))$ holds.

\begin{theorem} \label{Th-6-3}
Let $a,b,c,d \in \mathbb{N}$. If $a+b+c+d$ is odd then 
$\zeta_2 ((a,b,c,d);PSp(2))$ can be expressed as a polynomial in $\{\zeta(j+1)\,|\,j\in \mathbb{N}\}$ with $\mathbb{Q}$-coefficients.
\end{theorem}

In this case, we cannot directly obtain the assertion from Theorem \ref{T-PSp-2} unlike the case of $PU(3)$. In fact, even if we use \eqref{C2-Fq}, it seems unable to obtain an expression of $\zeta_2((1,2,2,2);PSp(2))$ because this value vanishes if we set $(p,s,q,r)=(1,2,2,2)$ or $(2,2,2,1)$ in \eqref{C2-Fq}.
Hence we use another method as follows. First we quote the following.

\begin{lemma}[\cite{TsRama},\,Theorem 4.1] \label{L-rama}
Let 
\begin{align}
\mathcal{T}_{\vv,\mu}(k,l,d)&=\sum_{l\geq 0 \atop m\geq 0}\frac{1}{(2l+\vv)^a (2m+\mu)^b (2l+2m+\vv+\mu)^c} \label{T-zeta}
\end{align}
for $k,l,d\in \mathbb{N}$ and $\vv,\mu\in \{1,2\}$. Suppose $k+l+d$ is odd, then $\mathcal{T}_{\vv,\mu}(k,l,d)$ can be expressed as a polynomial in $\{ \zeta(j+1)\,|\,j\in \mathbb{N}\}$ with $\mathbb{Q}$-coefficients.
\end{lemma}

It should be noted that the assertion in \cite[Theorem 4.1]{TsRama} includes a condition $d\geq 2$. However, by examining its proof, we can remove this condition. More precisely, we know that \cite[Theorem 4.1]{TsRama} can be derived from \cite[Theorem 3.4]{TsRama} which includes a condition $d\geq 2$. We can easily check that \cite[Theorem 3.4]{TsRama} holds for $d=1$ if we interpret the empty sum as $0$ in its statement. Thus \cite[Theorem 4.1]{TsRama} holds for $d=1$ which implies the above lemma. By this lemma we can prove Theorem \ref{Th-6-3} as follows. 

\begin{proof}[Proof of Theorem \ref{Th-6-3}] 
First we use the relation
\begin{align}
\frac{(-1)^c}{X^c(X+Y)^d}& =\sum_{j=1}^{c}\binom{c+d-j-1}{c-j}(-1)^j\frac{1}{Y^{c+d-j}X^j} \label{PFD}\\
& \quad +\sum_{j=1}^{d}\binom{c+d-j-1}{d-j}\frac{1}{Y^{c+d-j}(X+Y)^j}\notag
\end{align}
for $c,d\in \mathbb{N}$, which can be elementarily proved by induction on $c+d$ by using the partial fraction decomposition repeatedly. Therefore, setting $(X,Y)=(2l+1+m,m)$ in \eqref{PFD}, we see that 
\begin{align}
 & (-1)^c\zeta_2(a,b,c,d;PSp(2))= \sum_{l\geq 0 \atop m\geq 1} \frac{1}{(2l+1)^{a}m^{b}(2l+1+m)^{c}(2l+1+2m)^d}\label{T-C2}\\
& \ = \sum_{j=1}^{c}\binom{c+d-j-1}{c-j}(-1)^j \sum_{l\geq 0 \atop m\geq 1} \frac{1}{(2l+1)^{a}m^{b+c+d-j}(2l+1+m)^{j}}\notag\\
& \quad + \sum_{j=1}^{d}\binom{c+d-j-1}{d-j} \sum_{l\geq 0 \atop m\geq 1} \frac{1}{(2l+1)^{a}m^{b+c+d-j}(2l+1+2m)^{j}}\notag\\
& \ =\sum_{j=1}^{c}\binom{c+d-j-1}{c-j}(-1)^j \left\{\mathfrak{T}_{1,1}(a,b+c+d-j,j)+\mathfrak{T}_{1,2}(a,b+c+d-j,j)\right\}\notag\\
& \quad +\sum_{j=1}^{d}\binom{c+d-j-1}{d-j} 2^{b+c+d-j}\mathfrak{T}_{1,2}(a,b+c+d-j,j).\notag
\end{align}
Hence, by Lemma \ref{L-rama}, we obtain the assertion.
\end{proof}

\begin{example} \label{Exam-LL}
As we noted above, it seems to be unable to obtain an expression of $\zeta_2((1,2,2,2);PSp(2))$ in terms of $\zeta(s)$, from \eqref{C2-Fq}. Hence we use \eqref{T-C2}. Then we have 
\begin{align}
\zeta_2((1,2,2,2);PSp(2)) &= -2\mathfrak{T}_{1,1}(1,5,1)+62\mathfrak{T}_{1,2}(1,5,1) \label{ps1222}\\
& \quad+\mathfrak{T}_{1,1}(1,4,2)+17\mathfrak{T}_{1,2}(1,4,2).\notag
\end{align}
By the method used in \cite[Section 4]{TsRama}, we can obtain 
\begin{align*}
\mathfrak{T}_{1,1}(1,5,1)&=-\frac{105}{128}\zeta(3)\zeta(4)-\frac{93}{128}\zeta(5)\zeta(2)+\frac{381}{128}\zeta(7),\\
\mathfrak{T}_{1,2}(1,5,1)&=-\frac{7}{128}\zeta(3)\zeta(4)-\frac{31}{128}\zeta(5)\zeta(2)+\frac{127}{256}\zeta(7),\\
\mathfrak{T}_{1,1}(1,4,2)&=\frac{105}{128}\zeta(3)\zeta(4)+\frac{279}{128}\zeta(5)\zeta(2)-\frac{1143}{256}\zeta(7),\\
\mathfrak{T}_{1,1}(1,4,2)&=\frac{7}{128}\zeta(3)\zeta(4)+\frac{183}{128}\zeta(5)\zeta(2)-\frac{635}{256}\zeta(7).
\end{align*}
Substituting these results into \eqref{ps1222}, we obtain
$$\zeta_2((1,2,2,2);PSp(2)) =\frac{827}{64}\zeta(5)\zeta(2)-\frac{1397}{64}\zeta(7).$$
\end{example}

\bigskip

\noindent
\begin{Large}{\bf Appendix} \end{Large}

\ 

\noindent
As stated at the end of Section \ref{sec-4}, here we give the explicit form of $\mathcal{P}((2,4,4,2,\bs{y};C_2)$ as follows:
\begin{align*}
\mathcal{P}&((2,4,4,2,\bs{y};C_2)\\
=&
-\frac{187169}{435891456000}
-\frac{\{y_1\}}{684288}
+\frac{1577 \{y_1\}^2}{479001600}
+\frac{\{y_1\}^3}{103680}
-\frac{157 \{y_1\}^4}{14515200}
-\frac{\{y_1\}^5}{51840}
+\frac{613 \{y_1\}^6}{43545600}
\\
&
+\frac{47 \{y_1\}^7}{2419200}
-\frac{127 \{y_1\}^8}{14515200}
-\frac{89 \{y_1\}^9}{4354560}
+\frac{901 \{y_1\}^{10}}{43545600}
-\frac{53 \{y_1\}^{11}}{7257600}
+\frac{443 \{y_1\}^{12}}{479001600}
\\
&
-\frac{19 \{2 y_1-y_2\}}{47900160}
+\frac{347 \{2 y_1-y_2\}^2}{958003200}
+\frac{\{2 y_1-y_2\}^3}{14515200}
-\frac{\{2 y_1-y_2\}^4}{29030400}
-\frac{\left\{y_1-\frac{y_2}{2}\right\}}{748440}
\\
&
+\frac{\left\{y_1-\frac{y_2}{2}\right\}^2}{748440}
-\frac{\left\{\frac{1}{2}+y_1-\frac{y_2}{2}\right\}}{748440}
+\frac{\left\{\frac{1}{2}+y_1-\frac{y_2}{2}\right\}^2}{748440}
+\frac{\left\{\frac{1}{2}+\frac{y_2}{2}\right\}^2}{1069200}
\\
&
+\frac{\left\{\frac{1}{2}+y_1-\frac{y_2}{2}\right\} \left\{\frac{1}{2}+\frac{y_2}{2}\right\}^2}{37800}
-\frac{\left\{\frac{1}{2}+y_1-\frac{y_2}{2}\right\}^2 \left\{\frac{1}{2}+\frac{y_2}{2}\right\}^2}{37800}
-\frac{\left\{\frac{1}{2}+\frac{y_2}{2}\right\}^4}{340200}
\\
&
-\frac{\left\{\frac{1}{2}+y_1-\frac{y_2}{2}\right\} \left\{\frac{1}{2}+\frac{y_2}{2}\right\}^4}{11340}
+\frac{\left\{\frac{1}{2}+y_1-\frac{y_2}{2}\right\}^2 \left\{\frac{1}{2}+\frac{y_2}{2}\right\}^4}{11340}
+\frac{\left\{\frac{1}{2}+\frac{y_2}{2}\right\}^6}{340200}
\\
&
+\frac{\left\{\frac{1}{2}+y_1-\frac{y_2}{2}\right\} \left\{\frac{1}{2}+\frac{y_2}{2}\right\}^6}{8100}
-\frac{\left\{\frac{1}{2}+y_1-\frac{y_2}{2}\right\}^2 \left\{\frac{1}{2}+\frac{y_2}{2}\right\}^6}{8100}
+\frac{\left\{\frac{1}{2}+\frac{y_2}{2}\right\}^8}{226800}
\\
&
-\frac{\left\{\frac{1}{2}+y_1-\frac{y_2}{2}\right\} \left\{\frac{1}{2}+\frac{y_2}{2}\right\}^8}{7560}
+\frac{\left\{\frac{1}{2}+y_1-\frac{y_2}{2}\right\}^2 \left\{\frac{1}{2}+\frac{y_2}{2}\right\}^8}{7560}
-\frac{\left\{\frac{1}{2}+\frac{y_2}{2}\right\}^9}{68040}
\\
&
+\frac{\left\{\frac{1}{2}+y_1-\frac{y_2}{2}\right\} \left\{\frac{1}{2}+\frac{y_2}{2}\right\}^9}{11340}
-\frac{\left\{\frac{1}{2}+y_1-\frac{y_2}{2}\right\}^2 \left\{\frac{1}{2}+\frac{y_2}{2}\right\}^9}{11340}
+\frac{\left\{\frac{1}{2}+\frac{y_2}{2}\right\}^{10}}{68040}
\\
&
-\frac{\left\{\frac{1}{2}+y_1-\frac{y_2}{2}\right\} \left\{\frac{1}{2}+\frac{y_2}{2}\right\}^{10}}{56700}
+\frac{\left\{\frac{1}{2}+y_1-\frac{y_2}{2}\right\}^2 \left\{\frac{1}{2}+\frac{y_2}{2}\right\}^{10}}{56700}
-\frac{\left\{\frac{1}{2}+\frac{y_2}{2}\right\}^{11}}{155925}
\\
&
+\frac{\left\{\frac{1}{2}+\frac{y_2}{2}\right\}^{12}}{935550}
+\frac{\left\{\frac{y_2}{2}\right\}^2}{1069200}
+\frac{\left\{y_1-\frac{y_2}{2}\right\} \left\{\frac{y_2}{2}\right\}^2}{37800}
-\frac{\left\{y_1-\frac{y_2}{2}\right\}^2 \left\{\frac{y_2}{2}\right\}^2}{37800}
-\frac{\left\{\frac{y_2}{2}\right\}^4}{340200}
\\
&
-\frac{\left\{y_1-\frac{y_2}{2}\right\} \left\{\frac{y_2}{2}\right\}^4}{11340}
+\frac{\left\{y_1-\frac{y_2}{2}\right\}^2 \left\{\frac{y_2}{2}\right\}^4}{11340}
+\frac{\left\{\frac{y_2}{2}\right\}^6}{340200}
+\frac{\left\{y_1-\frac{y_2}{2}\right\} \left\{\frac{y_2}{2}\right\}^6}{8100}
\\
&
-\frac{\left\{y_1-\frac{y_2}{2}\right\}^2 \left\{\frac{y_2}{2}\right\}^6}{8100}
+\frac{\left\{\frac{y_2}{2}\right\}^8}{226800}
-\frac{\left\{y_1-\frac{y_2}{2}\right\} \left\{\frac{y_2}{2}\right\}^8}{7560}
+\frac{\left\{y_1-\frac{y_2}{2}\right\}^2 \left\{\frac{y_2}{2}\right\}^8}{7560}
-\frac{\left\{\frac{y_2}{2}\right\}^9}{68040}
\\
&
+\frac{\left\{y_1-\frac{y_2}{2}\right\} \left\{\frac{y_2}{2}\right\}^9}{11340}
-\frac{\left\{y_1-\frac{y_2}{2}\right\}^2 \left\{\frac{y_2}{2}\right\}^9}{11340}
+\frac{\left\{\frac{y_2}{2}\right\}^{10}}{68040}
-\frac{\left\{y_1-\frac{y_2}{2}\right\} \left\{\frac{y_2}{2}\right\}^{10}}{56700}
\\
&
+\frac{\left\{y_1-\frac{y_2}{2}\right\}^2 \left\{\frac{y_2}{2}\right\}^{10}}{56700}
-\frac{\left\{\frac{y_2}{2}\right\}^{11}}{155925}
+\frac{\left\{\frac{y_2}{2}\right\}^{12}}{935550}
-\frac{19 \{-2 y_1+y_2\}}{47900160}
+\frac{71 \{y_1\} \{-2 y_1+y_2\}}{29937600}
\\
&
+\frac{19 \{y_1\}^2 \{-2 y_1+y_2\}}{2419200}
-\frac{17 \{y_1\}^3 \{-2 y_1+y_2\}}{1088640}
-\frac{19 \{y_1\}^4 \{-2 y_1+y_2\}}{725760}
\\
&
+\frac{\{y_1\}^5 \{-2 y_1+y_2\}}{32400}
+\frac{19 \{y_1\}^6 \{-2 y_1+y_2\}}{518400}
-\frac{\{y_1\}^7 \{-2 y_1+y_2\}}{36288}
\\
&
-\frac{23 \{y_1\}^8 \{-2 y_1+y_2\}}{483840}
+\frac{131 \{y_1\}^9 \{-2 y_1+y_2\}}{2177280}
-\frac{89 \{y_1\}^{10} \{-2 y_1+y_2\}}{3628800}
\\
&
+\frac{\{y_1\}^{11} \{-2 y_1+y_2\}}{285120}
+\frac{347 \{-2 y_1+y_2\}^2}{958003200}
+\frac{\{y_1\} \{-2 y_1+y_2\}^2}{604800}
\\
&
-\frac{13 \{y_1\}^2 \{-2 y_1+y_2\}^2}{1814400}
-\frac{\{y_1\}^3 \{-2 y_1+y_2\}^2}{90720}
+\frac{23 \{y_1\}^4 \{-2 y_1+y_2\}^2}{967680}
\\
&
+\frac{\{y_1\}^5 \{-2 y_1+y_2\}^2}{43200}
-\frac{11 \{y_1\}^6 \{-2 y_1+y_2\}^2}{345600}
-\frac{\{y_1\}^7 \{-2 y_1+y_2\}^2}{26880}
\\
&
+\frac{\{y_1\}^8 \{-2 y_1+y_2\}^2}{15360}
-\frac{23 \{y_1\}^9 \{-2 y_1+y_2\}^2}{725760}
+\frac{19 \{y_1\}^{10} \{-2 y_1+y_2\}^2}{3628800}
\\
&
+\frac{\{-2 y_1+y_2\}^3}{14515200}
-\frac{\{y_1\} \{-2 y_1+y_2\}^3}{907200}
-\frac{\{y_1\}^2 \{-2 y_1+y_2\}^3}{725760}
+\frac{\{y_1\}^3 \{-2 y_1+y_2\}^3}{136080}
\\
&
+\frac{\{y_1\}^4 \{-2 y_1+y_2\}^3}{207360}
-\frac{\{y_1\}^5 \{-2 y_1+y_2\}^3}{64800}
-\frac{\{y_1\}^6 \{-2 y_1+y_2\}^3}{103680}
\\
&
+\frac{11 \{y_1\}^7 \{-2 y_1+y_2\}^3}{362880}
-\frac{\{y_1\}^8 \{-2 y_1+y_2\}^3}{53760}
+\frac{\{y_1\}^9 \{-2 y_1+y_2\}^3}{272160}
\\
&
-\frac{\{-2 y_1+y_2\}^4}{29030400}
+\frac{\{y_1\}^2 \{-2 y_1+y_2\}^4}{1451520}
-\frac{\{y_1\}^4 \{-2 y_1+y_2\}^4}{414720}
\\
&
+\frac{\{y_1\}^6 \{-2 y_1+y_2\}^4}{207360}
-\frac{\{y_1\}^7 \{-2 y_1+y_2\}^4}{241920}
+\frac{\{y_1\}^8 \{-2 y_1+y_2\}^4}{967680}
-\frac{\{-y_1+y_2\}}{684288}
\\
&
+\frac{71 \{2 y_1-y_2\} \{-y_1+y_2\}}{29937600}
+\frac{\{2 y_1-y_2\}^2 \{-y_1+y_2\}}{604800}
-\frac{\{2 y_1-y_2\}^3 \{-y_1+y_2\}}{907200}
\\
&
+\frac{1577 \{-y_1+y_2\}^2}{479001600}
+\frac{19 \{2 y_1-y_2\} \{-y_1+y_2\}^2}{2419200}
-\frac{13 \{2 y_1-y_2\}^2 \{-y_1+y_2\}^2}{1814400}
\\
&
-\frac{\{2 y_1-y_2\}^3 \{-y_1+y_2\}^2}{725760}
+\frac{\{2 y_1-y_2\}^4 \{-y_1+y_2\}^2}{1451520}
+\frac{\{-y_1+y_2\}^3}{103680}
\\
&
-\frac{17 \{2 y_1-y_2\} \{-y_1+y_2\}^3}{1088640}
-\frac{\{2 y_1-y_2\}^2 \{-y_1+y_2\}^3}{90720}
+\frac{\{2 y_1-y_2\}^3 \{-y_1+y_2\}^3}{136080}
\\
&
-\frac{157 \{-y_1+y_2\}^4}{14515200}
-\frac{19 \{2 y_1-y_2\} \{-y_1+y_2\}^4}{725760}
+\frac{23 \{2 y_1-y_2\}^2 \{-y_1+y_2\}^4}{967680}
\\
&
+\frac{\{2 y_1-y_2\}^3 \{-y_1+y_2\}^4}{207360}
-\frac{\{2 y_1-y_2\}^4 \{-y_1+y_2\}^4}{414720}
-\frac{\{-y_1+y_2\}^5}{51840}
\\
&
+\frac{\{2 y_1-y_2\} \{-y_1+y_2\}^5}{32400}
+\frac{\{2 y_1-y_2\}^2 \{-y_1+y_2\}^5}{43200}
-\frac{\{2 y_1-y_2\}^3 \{-y_1+y_2\}^5}{64800}
\\
&
+\frac{613 \{-y_1+y_2\}^6}{43545600}
+\frac{19 \{2 y_1-y_2\} \{-y_1+y_2\}^6}{518400}
-\frac{11 \{2 y_1-y_2\}^2 \{-y_1+y_2\}^6}{345600}
\\
&
-\frac{\{2 y_1-y_2\}^3 \{-y_1+y_2\}^6}{103680}
+\frac{\{2 y_1-y_2\}^4 \{-y_1+y_2\}^6}{207360}
+\frac{47 \{-y_1+y_2\}^7}{2419200}
\\
&
-\frac{\{2 y_1-y_2\} \{-y_1+y_2\}^7}{36288}
-\frac{\{2 y_1-y_2\}^2 \{-y_1+y_2\}^7}{26880}
+\frac{11 \{2 y_1-y_2\}^3 \{-y_1+y_2\}^7}{362880}
\\
&
-\frac{\{2 y_1-y_2\}^4 \{-y_1+y_2\}^7}{241920}
-\frac{127 \{-y_1+y_2\}^8}{14515200}
-\frac{23 \{2 y_1-y_2\} \{-y_1+y_2\}^8}{483840}
\\
&
+\frac{\{2 y_1-y_2\}^2 \{-y_1+y_2\}^8}{15360}
-\frac{\{2 y_1-y_2\}^3 \{-y_1+y_2\}^8}{53760}
+\frac{\{2 y_1-y_2\}^4 \{-y_1+y_2\}^8}{967680}
\\
&
-\frac{89 \{-y_1+y_2\}^9}{4354560}
+\frac{131 \{2 y_1-y_2\} \{-y_1+y_2\}^9}{2177280}
-\frac{23 \{2 y_1-y_2\}^2 \{-y_1+y_2\}^9}{725760}
\\
&
+\frac{\{2 y_1-y_2\}^3 \{-y_1+y_2\}^9}{272160}
+\frac{901 \{-y_1+y_2\}^{10}}{43545600}
-\frac{89 \{2 y_1-y_2\} \{-y_1+y_2\}^{10}}{3628800}
\\
&
+\frac{19 \{2 y_1-y_2\}^2 \{-y_1+y_2\}^{10}}{3628800}
-\frac{53 \{-y_1+y_2\}^{11}}{7257600}
+\frac{\{2 y_1-y_2\} \{-y_1+y_2\}^{11}}{285120}
\\
&
+\frac{443 \{-y_1+y_2\}^{12}}{479001600}.
\end{align*}

\

\ 

\ 

\begin{flushleft}
\begin{small}
{Y. Komori}: 
{Department of Mathematics, Rikkyo University, Nishi-Ikebukuro, Toshima-ku, Tokyo 171-8501, Japan}

e-mail: {komori@rikkyo.ac.jp}

\ 

{K. Matsumoto}: 
{Graduate School of Mathematics, Nagoya University, Chikusa-ku, Nagoya 464-8602 Japan}

e-mail: {kohjimat@math.nagoya-u.ac.jp}

\ 

{H. Tsumura}: 
{Department of Mathematics and Information Sciences, Tokyo Metropolitan University, 1-1, Minami-Ohsawa, Hachioji, Tokyo 192-0397 Japan}

e-mail: {tsumura@tmu.ac.jp}
\end{small}
\end{flushleft}


\begin{thebibliography}{999}

\bibitem{Ap}
T. M. Apostol, {Introduction to Analytic Number Theory}, Springer, 1976.

\bibitem{ApVu}
T. M. Apostol and T. H. Vu, Dirichlet series related to the Riemann zeta function, J. Number Theory {\bf 19} (1982), 85-102. 

\bibitem{Bourbaki}
{N. Bourbaki,} {Groupes et Alg{\`e}bres de Lie, Chapitres 4, 5 et 6}, Hermann, Paris, 1968.

\bibitem {D}
{K. Dilcher,} {Zeros of Bernoulli, generalized Bernoulli and Euler polynomials},
  Memoirs of Amer. Math. Soc. No. 386, 1988.

\bibitem {GS}
{P. E. Gunnells and R. Sczech}, 
Evaluation of Dedekind sums, Eisenstein cocycles, and special values of $L$-functions, {Duke Math.~J.} \textbf{118} (2003), 229--260.

% \bibitem{Ha}
% G. H. Hardy, Notes on some points in the integral calculus LV, On the integration of Fourier series, Messenger of Math. {51} (1922), 186--192. Reprinted in Collected papers of G. H. Hardy (including joint papers with J. E. Littlewood and others), Vol. III, edited by a committee appointed by the London Mathematical Society Clarendon Press, Oxford, 1969, pp.~506--512.

\bibitem{Hum72}
J. E. Humphreys, {Introduction to Lie Algebras and Representation Theory},
 Graduate Texts in Mathematics, Vol.~9, Springer-Verlag, New York-Berlin, 1972.

\bibitem{Hum}
J. E. Humphreys, {Reflection Groups and Coxeter Groups}, Cambridge University Press, Cambridge, 1990. 

\bibitem{IKZ}
K. Ihara, M. Kaneko, and D. Zagier, Derivation and double shuffle relations for multiple zeta values, Compositio Math. {\bf 142} (2006), 307-338.

%\bibitem{Kac}
%V. G. Kac, {Infinite-dimensional Lie algebras}, 3rd ed., Cambridge University Press, Cambridge, 1990. 

%\bibitem{HWZ}
%J. G. Huard, K. S. Williams and N.-Y. Zhang, {On Tornheim's double series}, {Acta Arith.} {\bf 75} (1996),\ 105-117.

\bibitem{Kaneko}
M. Kaneko, Multiple zeta values, Sugaku Expositions {\bf 18} (2005), 221-232. (Translation of Japanese original that appeared in S\=ugaku {\bf 54} (2002), 404-415.) 


\bibitem{Kn}
A. W. Knapp, {Lie Groups Beyond an Introduction}, 2nd ed. Birkh{\"a}user, 2002. 

\bibitem{Komori09}
Y. Komori, 
{An integral representation of multiple Hurwitz-Lerch zeta functions and generalized multiple Bernoulli numbers}, {Quart. J. Math. (Oxford)}, \textbf{61} (2010), 437-496. 

\bibitem{KMT}
{Y. Komori, K. Matsumoto and H. Tsumura,} 
Zeta-functions of root systems, in {The Conference on $L$-functions} (Fukuoka, 2006), L. Weng and M. Kaneko (eds.), World Scientific, 2007, pp.~115--140.

\bibitem{KMTpja}
{Y. Komori, K. Matsumoto and H. Tsumura,} 
Zeta and $L$-functions and Bernoulli polynomials of root systems, 
Proc. Japan Acad. Ser. A \textbf{84} (2008), 57--62.

\bibitem{KMTJC}
{Y. Komori, K. Matsumoto and H. Tsumura,} 
Functional relations for zeta-functions of root systems, in Number Theory; 
Dreaming in Dreams - Proceedings of the 5th China-Japan Seminar, T. Aoki, S. Kanemitsu and J. -Y. Liu (eds.), World Scientific Publ, 2010, pp.~135--183. 

\bibitem{KM5}
{Y. Komori, K. Matsumoto and H. Tsumura,} 
On multiple Bernoulli polynomials and multiple $L$-functions of root systems, {Proc. London Math. Soc.} \textbf{100} (2010), 303--347.

\bibitem{KMT-Mem}
{Y. Komori, K. Matsumoto and H. Tsumura,} 
An introduction to the theory of zeta-functions of root systems, in {Algebraic and Analytic Aspects of Zeta Functions and $L$-functions}, G. Bhowmik, K. Matsumoto and H. Tsumura (eds.), MSJ Memoirs, Vol. 21, Mathematical Society of Japan, 2010, pp.~115--140. 

\bibitem{KM2}
{Y. Komori, K. Matsumoto and H. Tsumura,} 
On Witten multiple zeta-functions associated with semisimple Lie algebras II, {J. Math. Soc. Japan} \textbf{62} (2010), 355--394. 

\bibitem{KM3}
{Y. Komori, K. Matsumoto and H. Tsumura,} 
On Witten multiple zeta-functions associated with semisimple Lie algebras III, to appear in The Proceedings of the Conference on Multiple Dirichlet Series and Applications to Automorphic Forms (Edinburgh, 2008), D. Bump et al. (eds.), arXiv:0907.0955. 

\bibitem{KM4}
{Y. Komori, K. Matsumoto and H. Tsumura,} 
On Witten multiple zeta-functions associated with semisimple Lie algebras IV, Glasgow Math. J. \textbf{53} (2011), 185-206.

% \bibitem{Ma} 
% K. Matsumoto, Asymptotic expansions of double zeta-functions of Barnes, of Shintani, and Eisenstein series, Nagoya Math J. \textbf{172} (2003), 59--102.

% \bibitem{Ma1} 
% K. Matsumoto, The analytic continuation and the asymptotic behaviour of certain multiple zeta-functions I, J. Number Theory, \textbf{101} (2003), 223--243.

\bibitem{MNO}
K. Matsumoto, T. Nakamura, H. Ochiai and H. Tsumura, 
On value-relations, functional relations and singularities of Mordell-Tornheim and related triple zeta-functions, {Acta Arith.} \textbf{132} (2008), 99--125.

\bibitem{MT}
{K. Matsumoto and H. Tsumura}, 
On Witten multiple zeta-functions associated with semisimple Lie algebras I, {Ann. Inst. Fourier} \textbf{56} (2006), 1457--1504.

% \bibitem{Na2}
% T. Nakamura, Double $L$-value relations and functional relations for Witten zeta functions, preprint.

\bibitem{Na3}
{T. Nakamura}, Double Lerch value relations and functional relations for Witten zeta functions, Tokyo J. Math. {\bf 31} (2008), 551-574. 

% \bibitem{Sa}
% H. Samelson, {Notes on Lie Algebras}, 
% Universitext, Springer-Verlag, 1990. 

\bibitem{SS}
M. V. Subbarao and R. Sitaramachandrarao, 
On some infinite series of L. J. Mordell and their analogues, 
{Pacific J. Math.} \textbf{119} (1985), 245--255. 

\bibitem{Sz98}
A. Szenes, 
Iterated residues and multiple Bernoulli polynomials, 
{Internat. Math. Res. Notices}, \textbf{18} (1998), 937--958.

\bibitem{Sz03}
A. Szenes, 
Residue formula for rational trigonometric sums, 
{Duke Math. J.} \textbf{118} (2003), 189--228. 

\bibitem{To}
L. Tornheim, Harmonic double series, Amer.~J. Math. \textbf{72} (1950), 303--314.

\bibitem{TPAMS}
{H. Tsumura, } 
On alternating analogues of Tornheim's double series, {Proc. Amer. Math. Soc.} \textbf{131} (2003), 3633--3641.

\bibitem{TMCOMP}
{H. Tsumura, } 
Evaluation formulas for Tornheim's type of alternating double series, {Math. Comp.} \textbf{73} (2004), 251--258.

\bibitem{TsArch}
{H. Tsumura}, 
{On Witten's type of zeta values attached to $SO(5)$}, Arch. Math. (Basel) \textbf{84} (2004), 147-152.

\bibitem{TsActa}
H. Tsumura, Combinatorial relations for Euler-Zagier sums, Acta Arith. {\bf 111} (2004), 27-42. 

\bibitem{TsPAMS}
H. Tsumura, On Mordell-Tornheim zeta values, Proc. Amer. Math. Soc. {\bf 133} (2005), 2387-2393. 

\bibitem{TsC}
{H. Tsumura, } 
On functional relations between the Mordell-Tornheim double zeta functions and the Riemann zeta function, {Math.~Proc.~Cambridge Philos. Soc.}, \textbf{142} (2007), 395--405.

\bibitem{TsRama}
{H. Tsumura, }
On alternating analogues of Tornheim's double series II, Ramanujan J. {\bf 18} (2009), 81-90.

\bibitem{Ve}
E. Verlinde,  
Fusion rules and modular transformations in $2D$ conformal field theory, 
{Nucl. Phys.} {\bf B300} (1988), 360--376.

\bibitem{WW}
E. T. Whittaker and G. N. Watson, 
{A Course of Modern Analysis}, 4th ed., Cambridge University Press, Cambridge, 1927.


\bibitem{Wi}
E. Witten, 
On quantum gauge theories in two dimensions, {Comm.~Math.~Phys.} \textbf{141} (1991), 153--209.

\bibitem{Za}
D. Zagier, 
Values of zeta functions and their applications, in {First European Congress of Mathematics} Vol.~II, A. Joseph
   et al. (eds.), Progr.~Math.~{120}, Birkh{\"a}user, 1994,
   pp.~497--512.

\end{thebibliography}
\end{document}